\documentclass[a4paper,11pt]{article}

\usepackage{graphicx}
\usepackage{amsmath}
\usepackage{amsfonts}
\usepackage{titlesec}
\usepackage{amsthm}
\usepackage{amssymb}
\usepackage{geometry}
\usepackage{color}
\newcommand{\R}{\mathbb{R}}
\newcommand{\T}{\mathcal{T}}

\newcommand{\gmin}{\gamma_{\operatorname{min}}}
\newcommand{\gmax}{\gamma_{\operatorname{max}}}
\newcommand{\dmin}{d_{\mathcal{U}}^{\mbox{\rm \tiny min}}}

\newcommand{\MsFEM}{\mbox{\tiny{MsFEM}}}

\newcommand{\Qh}{Q_h}
\newcommand{\Vf}{W_{h}}
\newcommand{\oVf}{\mathring{W}_{h}}
\newcommand{\Ic}{I_H}
\newcommand{\Icloc}[1]{I_H^{#1}}
\newtheorem{theorem}{Theorem}[section]
\newtheorem{corollary}[theorem]{Corollary}

\newtheorem{lemma}[theorem]{Lemma}

\theoremstyle{definition}
\newtheorem{definition}[theorem]{Definition}
\newtheorem{remark}[theorem]{Remark}
\newtheorem*{modelproblem}{Model Problem} %
\newtheorem{osstrategy}{Oversampling Strategy} %

\title{Oversampling for the Multiscale Finite Element Method}
\author{Patrick Henning \thanks{Institut f\"ur Numerische und Angewandte Mathematik, Westf\"alische Wilhelms-Universit\"at M\"unster, Einsteinstr. 62, D-48149 M\"unster, Germany}
 \and Daniel Peterseim \thanks{Institut f\"ur Mathematik, Humboldt-Universit\"at zu Berlin, Unter den Linden 6, D-10099 Berlin, Germany}
 \hspace{1ex}\thanks{The second author is supported by the DFG Research Center Matheon Berlin through project C33.}}

\begin{document}

\maketitle

\begin{abstract}
This paper reviews standard oversampling strategies as performed in the Multiscale Finite Element Method (MsFEM). Common to those approaches is that the oversampling is performed in the full space restricted to a patch but including coarse finite element functions. We suggest, by contrast, to perform local computations with the additional constraint that trial and test functions are linear independent from coarse finite element functions. This approach re-interprets
the Variational Multiscale Method in the context of computational homogenization.   
This connection gives rise to a general fully discrete error analysis for the proposed multiscale method with constrained oversampling without any resonance effects. In particular, we are able to give the first rigorous proof of convergence for a MsFEM with oversampling.
\end{abstract}



\section{Introduction}
\label{section-introduction}

The numerical treatment of partial differential equations with rapidly varying and strongly heterogeneous coefficient functions is still a challenging area of present research, especially with regard to applications such as porous media flow or the transport of solutes in groundwater. In such problems, the occurring permeabilities and hydraulic conductivities have rapidly changing features due to different types of soil, microscopic inclusions in the bottom or porous subsurface rock formations. Any meaningful numerical simulation of relevant physical effects has to account for these highly heterogeneous fine scale structures in the whole computational domain. This means that the underlying computational mesh has to be sufficiently fine to resolve microscopic details. If pore scale effects become relevant or if domains spread over kilometers, therefore, the computational load becomes extremely large and in several applications even too large to treat the problem with standard finite element or finite volume 
methods. This is just one instance of a so called multiscale problem as it arises in hydrology, physics or industrial engineering.

In recent years, many numerical methods have been designed to deal with these computational issues that come along with multiscale problems. Most of them aim to decouple the global fine scale problem into localized subproblems which can be treated independently from each other plus some global coarse problem. The list of meanwhile proposed multiscale methods is long. Amongst the most popular methods are the heterogeneous multiscale finite element method (HMM), initially introduced by E and Engquist \cite{E:Engquist:2003} (see also \cite{E:Engquist:2003-2,E:Engquist:2005,Abdulle:2009}, the Variational Multiscale Method by Hughes et al. \cite{Hughes:1995,Hughes:et-al:1998} (see also \cite{Larson:Malqvist:2007,Malqvist:2011,2011arXiv1110.0692M}) or the approaches by Owhadi and Zhang \cite{Owhadi:Zhang:2007,Owhadi:Zhang:2011}.

In this paper, we deal with another popular method: the Multiscale Finite Element Method (MsFEM) proposed by Hou and Wu \cite{Hou:Wu:1997} and further investigated in several contributions \cite{Hou:Wu:Cai:1999,Efendiev:Hou:Wu:2000,Efendiev:Pankov:2003,Hou:Wu:Zhang:2004}. There is an ongoing development of the method to apply it to various fields and equations. For instance, a MsFEM for nonlinear elliptic problems is proposed in \cite{Efendiev:Hou:Ginting:2004}, a formulation for two phase flow problems in porous media is presented in \cite{Efendiev:Hou:2007}, advection diffusion problems are treated in \cite{Deng:Yun:Xie:2009} and an application to elliptic interface problems with high contrast coefficients is presented in \cite{Chu:Graham:Hou:2010}. A survey on the method is given in the book by Efendiev and Hou \cite{Efendiev:Hou:2009}. There is vast literature of works devoted to the method but there are still open questions of strong interest. The most relevant issue is a rigorous error 
analysis of the method, in particular in the case of non-periodic microstructures. 

The MsFEM is related to some common finite element space with an underlying coarse grid. The essential idea is to modify the corresponding basis functions in such a way that fine scale variations on finer scales are sufficiently well captured. More specifically, local fine scale computations are performed to determine so called {\it corrector functions}. These corrector functions can be added as local perturbations to the original set of basis functions of the coarse finite element space.

However, it is well known that the classical MsFEM suffers from so called resonance errors,
which are typically of order $O(\frac{\varepsilon}{H})$, where $\varepsilon$ denotes a characteristic size of the small scale and where $H$ denotes the mesh size of the coarse grid (c.f. \cite{Hou:Wu:Cai:1999,Hou:Wu:Zhang:2004}).
This implies that the numerical error becomes large in regions where the coarse grid size is close to the characteristic length scale of the microscopic oscillations.
There are two different explanations for this error. The first one is a mismatch between the boundary conditions imposed for the local fine scale problems
and global behavior of the oscillatory exact solution (c.f \cite{Efendiev:Hou:Wu:2000}). The second explanation is due to size and geometry of the sampling patch (c.f. \cite{Hou:Wu:Zhang:2004}). The averaged behavior in such a patch should be 'representative' so that we can speak about a perfect sample size. If this is not the case, the final approximation might be distorted. In the periodic setting, for instance, the sampling domain should be some multiple of the periodic cell. On triangular patches with cathetuses of the length of a period, this patch is only half a periodic cell (i.e. the patch has bad size and geometry) and lacks essential information. This yields a completely wrong approximation (c.f. \cite{Henning:2012}). In the periodic setting considerable improvements were obtained by Gloria \cite{Gloria:2011,Gloria:2012}, who proposed a regularization of the local (patch) problems by adding a zero-order term. With this strategy, both sources of the oversampling error could be significantly reduced (
c.f. \cite[Theorem 3.1]{Gloria:2011} and \cite[Sections 5.3 and 5.4]{Gloria:2012}.

In a lot of applications, such as oil reservoir simulations or the transport of solutes in groundwater, a characteristic microscopic  length scale $\varepsilon$ is unknown, cannot be identified or does not exist at all. In scenarios without  a clear scale separation it is often impossible to predict whether or not we are in the regime of resonance errors. It is very likely to actually hit the problematic regime.
Hence, the quality of the final approximation can not be determined unless resonance errors are eliminated. For this purpose, different {\it oversampling strategies} have been proposed.
The fundamental idea of each of these techniques is to extend the local problems to larger patches, perform the computation on these {\it oversampling domains} but feed the coarse scale
equation only with the information obtained within the original smaller patches. This reduces the effect of wrong boundary conditions and bad sampling sizes. In this paper, we
present the two major strategies for oversampling and discuss their advantages and disadvantages. On the basis of these considerations we propose a new strategy that overcomes the issues of the existing strategies. The new approach is closely related to the VMM-type method presented in \cite{2011arXiv1110.0692M}.  We prove quantitative error estimates for the corresponding multiscale approximations under very general assumptions on the diffusion coefficient.

$\\$
This contribution is structured as follows: In Section \ref{section-msfem-general} we recall the classical formulation of the MsFEM without oversampling. The most popular approaches for oversampling are discussed in Section \ref{section-oversampling-strategies}. In Section \ref{section-new-oversampling-strategy} we propose a new strategy for which we present a quantitative error analysis. Numerical experiments are presented in Section \ref{section-numerical-experiments}. The paper closes with a short conclusion.

\section{The Multiscale Finite Element Method}
\label{section-msfem-general}

In this section, we state the setting of this paper and we establish the required notations. We recall the classical Multiscale Finite Element Method (MsFEM) as initially proposed by Hou and Wu \cite{Hou:Wu:1997}.

\subsection{Setting and notation}
\label{subsection-setting}
Consider a bounded Lipschitz domain $ \Omega\subset\mathbb{R}^{d}$ with piecewise flat boundary and some matrix-valued coefficient
$A\in L^\infty(\Omega,\mathbb{R}^{d\times d}_{sym})$ with uniform spectral bounds $\gmin>0$ and $\gmax\geq\gmin$, 
\begin{equation}\label{e:spectralbound}
\sigma(A(x))\subset [\gmin,\gmax]\quad\text{for almost all }x\in \Omega.
\end{equation}
Given $f\in L^{2}( \Omega) $, we seek the weak solution of 
\begin{equation*}\label{eq:model}
  \begin{aligned}
    -\nabla \cdot A\nabla u &= f\quad \text{in }\Omega, \\
    u & = 0 \quad \text{on }\partial\Omega, \\
  \end{aligned}
\end{equation*}
i.e., we seek $u\in H^1_0(\Omega):=\{v\in H^1(\Omega)\;\vert\;v\vert_{\partial\Omega}=0\text{ in the sense of traces}\}$ that satisfies
\begin{equation}\label{e:modelproblem}
a\left(  u,v\right):=\int_{\Omega}A\nabla u\cdot \nabla
v =\int_{\Omega}fv=:F( v) \quad\text{for all } v\in H^1_0(\Omega).
\end{equation}

We consider two discretization scales $H\geq h>0$. The coarse scale $H$ is arbitrary whereas the small
scale parameter $h$ may be constrained by the problem. Typically, it is assumed to be smaller than the characteristic length scales of the
variations of the diffusion coefficient $A$. 

Let $\T_H$, $\T_h$ denote corresponding subdivisions of $\Omega$ into (closed) triangles (for $d=2$) and tetrahedra (for $d=3$),
i.e., $\bar{\Omega}=\bigcup_{t\in\T_h}t=\bigcup{T\in\T_H}T$. We assume that $\T_H$, $\T_h$ are regular
in the sense that any two elements are either disjoint or share exactly one face or share exactly one edge or share exactly one vertex.
For simplicity we assume that $\T_h$ is derived from $\T_H$ by some regular, possibly non-uniform, mesh refinement.

For $\T=\T_H,\T_h$, let
\begin{equation*}
P_1(\T) = \{v \in C^0(\Omega) \;\vert \;\forall T\in\T,v\vert_T \text{ is a polynomial of total degree}\leq 1\}
\end{equation*}
denote the set of continuous and piecewise affine functions. 

Accordingly, $V_h:=P_1(\T_h)\cap H^1_0(\Omega)$ denotes the 'high resolution' finite element space and the 'coarse space'
is given by $V_H:=P_1(\T_H)\cap H^1_0(\Omega) \subset V_h$. For any given subset $\omega \subset \Omega$ we define the restriction
of $V_h$ to $\omega$ with a zero boundary condition by $\mathring{V}_h(\omega):=V_h \cap H^1_0(\omega)$. The nonconforming fine space $V_{h,\T_H}$ is defined by
\begin{align*}
 V_{h,\T_H} := \{ v_h\;|\; \hspace{2pt} \forall T \in \T_H, \hspace{2pt} (v_h)\vert_{T}\in V_h \cap H^1(T)\}.
\end{align*}
A general function in $v\in V_{h,\T_H}$ may jump across edges of the coarse mesh $\T_H$ and, hence, does not belong to $H^1_0(\Omega)$. However, the $\T_H$-piecewise gradient $\nabla_H v$, with $(\nabla_H v)\vert_T=\nabla (v\vert_T)$ for all $T\in\T_H$, exists. Typically, MsFEM approximations obtained with oversampling are nonconforming approximations of the exact solution in the sense that they do not belong to $H^1_0(\Omega)$.

In the following $x_T \in T$ denotes an arbitrary point, for instance the barycenter of $T$. For $\Phi_H \in V_H$ and $T \in \T_H$, the affine extension operator $E_T : V_H \rightarrow P_1(\Omega)$ is given by:
\begin{align*}
E_T(\Phi_H)(x):=(x-x_T)\cdot \nabla \Phi_H(x_T) + \Phi_H(x_T).
\end{align*}
Finally, by $\chi_T$ we denote the characteristic (or indicator) function with $\chi_T(x)=1$ for $x\in T$ and $\chi_T(x)=0$ elsewhere.

For the sake of simplicity, all fine scale computations are performed in subspaces of the finescale finite element space $V_h$. 
The Galerkin solution $u_h\in V_h$ which satisfies
\begin{equation}\label{e:modelref}
a_h(u_h,v) = F(v) \quad\text{ for all }v\in V_h
\end{equation}
may, hence, be considered as a reference approximation. 
Note that we never solve this large-scale equation. The function $u_h$ serves as a reference solution to compare our multiscale approximations with. The underlying assumption is that the mesh $\T_h$ is chosen sufficiently fine so that $u_h$ is sufficiently accurate. 

Throughout this paper, standard notation on Lebesgue and Sobolev spaces is employed and $a\lesssim b$ abbreviates an inequality $a\leq C\,b$
with some generic constant $0\leq C<\infty$ that may depend on the shape regularity of finite element meshes and the contrast $\gmax/\gmin$ but \emph{not} on the mesh sizes $H$, $h$, and the regularity or the variations of the diffusion matrix $A$; $a\approx b$ abbreviates $a\lesssim b\lesssim a$. 

\subsection{The classical MsFEM and reformulation}

We first present the classical MsFEM without oversampling as originally stated by Hou and Wu \cite{Hou:Wu:1997}, and similarly by Brezzi et al. \cite{Brezzi:Franca:Hughes:Russo:1997}.
They proposed the strategy to enrich the set of standard finite element basis functions by fine scale information.
The information is determined by solving local problems on the fine scale. We recall briefly the method and reformulate it in terms of a correction operator $Q_h$ and a corresponding corrector basis.

$\\$
Let $N$ denote the dimension of the coarse space $V_H$ and let $\{ \Phi_i\;|\; \enspace 1 \le i \le N\}$ denote the usual nodal basis of $V_H$.
Given some basis function $\Phi_i$, the corresponding MsFEM basis function $\Phi_i^{\MsFEM} \in V_h$ is uniquely determined by the condition that for all $T \in \T_H$ and for all $\phi_h \in \mathring{V}_h(T)$ it holds
\begin{align}
\label{multiscale-basis-functions} \int_{T} A(x) \nabla \Phi_i^{\MsFEM}(x) \cdot \nabla \phi_h(x)\hspace{2pt} dx = 0\quad\text{and }\Phi_i^{\MsFEM}=\Phi_i \text{ on }\partial T.
\end{align}
The span of these MsFEM functions is called the MsFEM solution space
\begin{align*}
 V_H^{\MsFEM} := \mbox{\rm span}\{ \Phi_i^{\MsFEM}\;|\; \enspace 1 \le i \le N\}.
\end{align*}
This space is conforming in the sense of $V_H^{\MsFEM} \subset V_h \subset H^1_0(\Omega)$, because $\{ \Phi_i^{\MsFEM}\;|\; \enspace 1 \le i \le N\}$ defines
a conforming set of basis functions.
The classical MsFEM in Petrov-Galerkin (PG) formulation due to \cite{Hou:Wu:Zhang:2004} reads as follows.
\begin{definition}[MsFEM without oversampling]
\label{definition-MsFEM-without-oversampling}The MsFEM approximation $u_H^{\MsFEM}\in V_H^{\MsFEM}$ is defined as the solution of
\begin{align}\label{e:msfem1}
 \int_{\Omega} A(x) \nabla u_H^{\MsFEM}(x) \cdot \nabla \Phi_H(x) \hspace{2pt} dx = \int_{\Omega} f(x) \Phi_H(x) \hspace{2pt} dx \qquad \mbox{for all} \enspace \Phi_H \in V_H.
\end{align}
\end{definition}
In \cite{Hou:Wu:1997}, the MsFEM was originally proposed in Galerkin formulation, i.e. the test functions $\Phi_H \in V_H$ in \eqref{e:msfem1} are replaced by test functions $\Phi_H^{\MsFEM}\in V_H^{\MsFEM}$. Observe that due to the orthogonality property (\ref{multiscale-basis-functions}) both formulations are almost identical (in the absence of oversampling). For structural reasons we used the PG version to introduce the MsFEM.

With regard to the general framework for oversampling that we present in the subsequent sections, we note that the MsFEM can be rewritten in the following way.
\begin{remark}
\label{reformulation-classical-msfem}
If $u_H^{\MsFEM}\in V_H^{\MsFEM}$ denotes the MsFEM approximation stated in Definition \ref{definition-MsFEM-without-oversampling}, then we have $u_H^{\MsFEM}=u_H+Q_h(u_H)$, where $u_H \in V_H$ solves
\begin{subequations}
 \label{e:msfem2}
\end{subequations}
\begin{align}
\label{def-local-probl-no-oversampling} \int_{\Omega} A \left( \nabla u_H + \nabla Q_h(u_H) \right) \cdot \nabla \Phi_H = \int_{\Omega} f \Phi_H\quad\text{for all }\Phi_H \in V_H,
\tag{\ref{e:msfem2}.a}
\end{align}
with
\begin{align}
Q_h(\Phi_H)(x):= \sum_{T \in \T_H} \sum_{i=1}^d \partial_{x_i}\Phi_H(x_T) w_{T,i}(x),
\tag{\ref{e:msfem2}.b}
\end{align}
and $w_{T,i} \in \mathring{V}_h(T)$ being the unique solution of
\begin{align}
\int_{T} A \nabla w_{T,i} \cdot \nabla \phi_h = - \int_{T} A e_i \cdot \nabla \phi_h
\quad \mbox{for all } \phi_h \in \mathring{V}_h(T).
\tag{\ref{e:msfem2}.c}
\end{align}
The set of all functions $w_{T,i}$ is what we are going to call a {\it local corrector basis}. From the computational point of view, it seems on first glance to be cheaper to compute the corrector basis given by ({\ref{e:msfem2}.c}) instead of directly computing the set of multiscale basis functions given by (\ref{multiscale-basis-functions}). The latter one formally involves more problems to solve. For instance if $d=2$, the assembling of the corrector basis $\{ w_{T,i}\;\vert\;T \in \T_H, i=1,2\}$ requires the solution of $2\cdot |\T_H|$ local problems, whereas the solutions of $3\cdot |\T_H|$ local problems are required to assemble $\{ \Phi_i^{\MsFEM}\;|\; \enspace 1 \le i \le N\}$ by using the gradients of coarse basis functions (for which we have 3 per coarse element). Still, it is possible to use the partition of unity property of the basis functions to equally decrease the costs of the original version of the MsFEM from $d \cdot |\T_H|$ to $(d-1)\cdot |\T_H|$. In particular, restricted to $T$, the gradients of $(d-1)$ basis functions 
associated with $(d-1)$ corners of the element $T$ span the gradient of the missing $d$'th basis function on $T$.
\end{remark}
The equivalence between the formulations \eqref{e:msfem1} and \eqref{e:msfem2} can be easily verified by the relation $\Phi_i^{\MsFEM} = \Phi_i + Q_h(\Phi_i)$. Observe that for every $i$, for every $T \in \T_H$ and every $\phi_h \in \mathring{V}_h(T)$,
\begin{eqnarray*}
\lefteqn{\int_{T} A(x) \left( \nabla \Phi_i(x) + \nabla Q_h(\Phi_i)(x) \right) \cdot \nabla \phi(x) \hspace{2pt} dx} \\
&=& \sum_{i=1}^d \partial_{x_i}\Phi_H(x_T) \int_{T} A(x) \left( e_i + \nabla w_T^i(x) \right) \cdot \nabla \phi(x) \hspace{2pt} dx = 0
\end{eqnarray*}
and $\Phi_i + Q_h(\Phi_i) = \Phi_i$ on $\partial T$, which is the definition of $\Phi_i^{\MsFEM}$. 

A symmetric formulation of \eqref{def-local-probl-no-oversampling} is given by: find $u_H \in V_H$ with
\begin{align}
\label{def-local-probl-no-oversampling-symmetric} \int_{\Omega} A \left( \nabla u_H + \nabla Q_h(u_H) \right) \cdot \left( \nabla \Phi_H + \nabla Q_h(\Phi_H) \right)= \int_{\Omega} f \Phi_H\quad\text{for all }\Phi_H \in V_H.
\end{align}
Note that \eqref{def-local-probl-no-oversampling} and \eqref{def-local-probl-no-oversampling-symmetric} are identical, because
\begin{align*}
\int_{T} A \left( \nabla u_H + \nabla Q_h(u_H) \right) \cdot \nabla \phi_h = 0 \quad \mbox{for all} \enspace \phi_h \in \mathring{V}_h(T).
\end{align*}

\section{Oversampling strategies}
\label{section-oversampling-strategies}

As already discussed in the introduction, the classical MsFEM in Definition \ref{definition-MsFEM-without-oversampling} can be strongly affected or even dominated by resonance errors (c.f. \cite{Efendiev:Hou:2009}). In the absence of scale separation or any knowledge about a suitable sample size for the local problems, the classical MsFEM needs a modification.
{\it Oversampling} is considered to be a remedy to this issue. Oversampling means that the local problems (\ref{def-local-probl-no-oversampling}) are solved on larger domains, but only the interior information (i.e. we restrict the gained fine scale information to $T$) is communicated to the coarse scale equation (\ref{def-local-probl-no-oversampling}). 

There is no unique way of extending the local problems (\ref{def-local-probl-no-oversampling}) to larger patches. Different extensions lead to different oversampling strategies. In this section, we present the two common approaches for oversampling. We rephrase both approaches so that they fit into a common framework. We discuss the advantages and disadvantages of the methods and then we propose our new oversampling strategy. Note that each of the subsequent strategies is a generalization of the case without oversampling.

$\\$
We shall introduce some additional notation. 
\begin{definition}[Admissible patch]
For $T \in \T_H$, we call $U(T)$ an {\it admissible patch} of
$T$, if it is non-empty, open, and connected, if $T \subset U(T) \subset \Omega$ and if
it is the union of elements of $\T_h$, i.e.
\begin{align*}
    U(T) = \operatorname{int}\bigcup_{\tau \in \T_h^{\ast}} \tau, \quad \mbox{where} \enspace
    \T_h^{\ast} \subset \T_h.
\end{align*}
\end{definition}

A given set of admissible patches is given by $\mathcal{U}$, i.e.
\begin{align*}
 \mathcal{U} := \{ U(T)\;|\; \hspace{2pt} T \in \T_H \enspace \mbox{and } U(T) \enspace \mbox{is an admissible patch}\},
\end{align*}
where $\mathcal{U}$ contains one and only one patch $U(T)$ for each $T \in \T_H$. The set $U(T) \setminus T$ is called an {\it oversampling layer}. The thickness of the oversampling layer is denoted by $d_{\mathcal{U},T}:=\mbox{dist}(T,\partial U(T))$. Furthermore, we define 
\begin{align*}
d_{\mathcal{U}}^{\mbox{\rm \tiny min}} := \min_{T \in \T_H} d_{\mathcal{U},T} \quad \mbox{and} \quad d_{\mathcal{U}}^{\mbox{\rm \tiny max}} := \max_{T \in \T_H} d_{\mathcal{U},T}
\end{align*}
the minimum and maximum thickness.

$\\$
In the spirit of \eqref{def-local-probl-no-oversampling} and \eqref{def-local-probl-no-oversampling-symmetric}, we now define the coarse scale equation for an arbitrary Multiscale Finite Element Method with a chosen oversampling strategy. As we will see later on, all MsFEM realizations only differ in the correction operator $Q_h$ that determines the oversampling strategy.
\begin{definition}[Framework for Oversampling Strategies]\label{d:framework}
\label{oversampling-framework}
Let $\alpha=1,2,3$ denote the index of the oversampling strategy to be specified later on and let 
\begin{align*}
 \{ w_{h,T,i}^{\mathcal{U},\alpha}\;|\; \hspace{2pt} 1 \le i \le d, \hspace{2pt} T \in \T_H \}
\end{align*}
denote a given local corrector basis that depends on the chosen strategy $\alpha$ (see (\ref{e:msfem2}.a)-(\ref{e:msfem2}.c) for the trivial case of such a basis). Then, a (not necessarily conforming) correction operator $Q_h^{\mathcal{U},\alpha} : V_H \rightarrow V_{h,\T_H}$ is defined by 
\begin{align}
\label{relation-corrector-and-basis} Q_h^{\mathcal{U},\alpha}(\Phi_H)(x):= \sum_{T \in \T_H} \chi_{T}(x) \sum_{i=1}^d \partial_{x_i} \Phi_H(x_T) \hspace{2pt} w_{h,T,i}^{\mathcal{U},\alpha}(x) \quad \mbox{for} \enspace \Phi_H \in V_H.
\end{align}
The MsFEM approximation $u_H^{\alpha} + Q_h^{\mathcal{U},\alpha}(u_H^{\alpha})$ obtained with strategy $\alpha$ in Petrov-Galerkin formulation reads: find $u_H^{\alpha} \in V_H$ such that
\begin{align}
\label{coarse-scale-equation-in-framework-pg}\sum_{T \in \T_H} \int_{T} A \left( \nabla u_H^{\alpha} + \nabla Q_h^{\mathcal{U},\alpha}(u_H^{\alpha}) \right) \cdot \nabla \Phi_H  = \int_{\Omega} f \Phi_H\quad\text{for all }\Phi_H \in V_H. 
\end{align}
The MsFEM approximation $u_H^{S,\alpha} + Q_h^{\mathcal{U},\alpha}(u_H^{S,\alpha})$ obtained with strategy $\alpha$ and a (not necessarily equivalent) symmetric formulation is given by: find $u_H^{\alpha,\operatorname{sym}} \in V_H$ with
\begin{align}
\label{symmetric-version-msfem}&\sum_{T \in \T_H} \int_{T} A \left( \nabla u_H^{\alpha,\operatorname{sym}} + \nabla Q_h^{\mathcal{U},\alpha}(u_H^{\alpha,\operatorname{sym}}) \right) \cdot \left( \nabla \Phi_H + \nabla Q_h^{\mathcal{U},\alpha}(\Phi_H) \right) \\
\nonumber&\qquad = \int_{\Omega} f (\Phi_H + Q_h^{\mathcal{U},\alpha}(\Phi_H))
\end{align}
for all $\Phi_H \in V_H$. Observe that strategies can {\it only} differ in the choice of the corrector basis. The remaining structure of the methods is always the same.
\end{definition}

In the subsequent sections, we demonstrate how existing oversampling strategies fit into the framework presented in Definition \ref{oversampling-framework}.

\subsection{Classical strategy initially introduced by Hou and Wu}

The classical oversampling strategy was proposed by Hou and Wu \cite{Hou:Wu:1997} and further used and investigated in several works (c.f. \cite{Efendiev:Hou:Ginting:2004,Chen:Savchuk:2007,Efendiev:Hou:2009}).

Let $T\in \T_H$ be fixed and let $\{\Phi^T_1,\Phi^T_2,...,\Phi^T_{d+1}\}\subset V_H$ denote the basis functions that belong to the $d+1$ nodal points in $T$. Hou and Wu \cite{Hou:Wu:1997} proposed the following oversampling strategy: solve for $\tilde{\Phi}^T_j \in V_h(U(T))$ with
\begin{align}
\label{strategy-hou-wu-local-problem} \int_{U(T)} A \nabla \tilde{\Phi}^T_j \cdot \nabla \phi_h = 0 \quad \mbox{for all } \phi_h \in \mathring{V}_h(U(T))
\end{align}
and the boundary condition $\tilde{\Phi}^T_j=E_T(\Phi^T_j)$ on $\partial U(T)$, where $E_T(\Phi^T_j)$ denotes the affine extension of $(\Phi^T_j)\vert_{T}$. Then, for a given coarse function $\Phi_H \in V_H$, $\Phi_H^{\MsFEM}$ is defined by
\begin{align*}
\Phi_H^{\MsFEM}=\sum_{j=1}^{d+1} c_j \tilde{\Phi}^T_j.
\end{align*}
where the $c_j$ are such that $\Phi_H^{\MsFEM}(x_j)=\Phi_H(x_j)$ for all  $d+1$ coarse nodes $x_j$ of $T$. The final coarse scale equation in Petrov-Galerkin formulation reads: find $u^{\mathbf{1},\mbox{\rm \tiny MsFEM}}_H \in V_{h,\T_H}$ with
\begin{align}
\label{classical-formulation-strategy-1} \int_\Omega A \nabla_H u^{\mathbf{1},\mbox{\rm \tiny MsFEM}}_H \cdot \nabla \Phi_H = \int_{\Omega} f \Phi_H
\end{align}
for all $\Phi_H \in V_H$. Observe that $u_H^{\MsFEM}$ is a nonconforming approximation of $u$. In \cite{Hou:Wu:Zhang:2004,Efendiev:Hou:Wu:2000}, a slightly different condition is used to define the coefficients $c_i$. However, it turns out that this modified condition leads to nothing but Oversampling Strategy 2 below.

We shall rephrase this multiscale method with oversampling strategy in the framework of Definition~\ref{d:framework}. Let $Q_T(\Phi_H):=\Phi_H^{\MsFEM}-E_T(\Phi_H)$ define the local corrector, i.e., an  operator that communicates fine scale information to the coarse scale equation. The corresponding reduced fine scale space $\mathring{V}_h^r(U(T))$ is given by
\begin{align}
\label{def-reduced-space-stragey-1}\mathring{V}_h^r(U(T)) := \mathring{V}_h(U(T)) \setminus \mbox{\rm span}\{\Phi^T_1,\Phi^T_2,...,\Phi^T_{d+1}\}
\end{align}
with nodal basis functions $\{\Phi^T_1,\Phi^T_2,...,\Phi^T_{d+1}\}\subset V_H$.
Since 
\begin{align*}
Q_T(\Phi_H)(x_i)=\Phi_H^{\MsFEM}(x_i)-\Phi_H(x_i)=0 \quad \mbox{for all nodes $x_i$ in $T$},
\end{align*}
$Q_T(\Phi_H) \in \mathring{V}_h^r(U(T))$. Moreover, by the definition of $\Phi_H^{\MsFEM}$, $Q_T(\Phi_H) \in \mathring{V}_h^r(U(T))$ satisfies
\begin{eqnarray*}
\lefteqn{\int_{U(T)} A  \left( \nabla \Phi_H(x_T) + \nabla Q_T(\Phi_H) \right) \cdot \nabla \phi_h}\\
&=& \int_{U(T)} A  \left( \nabla E_T(\Phi_H) + \nabla Q_T(\Phi_H) \right) \cdot \nabla \phi_h\\
&=& \sum_{i=1}^d c_i \int_{U(T)} A \nabla \tilde{\Phi}^T_i \cdot \nabla \phi_h = 0,
\end{eqnarray*}
for all $\phi_h \in \mathring{V}_h^r(U(T))$. Since $\nabla \Phi_H(x_T)$ is a constant in $U(T)$ we may rewrite $Q_T(\Phi_H)$ in terms of a corrector basis. This gives us the first definition of oversampling within our framework. 
\begin{osstrategy}
Let $\mathring{V}_h^r(U(T))$ denote the reduced fine scale space given by (\ref{def-reduced-space-stragey-1}) and let $w_{h,T,i}^{\mathcal{U},\mathbf{1}} \in \mathring{V}_h^r(U(T))$ (for $i \in \{ 1,2,..,d\}$) denote the solution of
\begin{align}
\label{corrector-os1}\int_{U(T)} A \nabla w_{h,T,i}^{\mathcal{U},\mathbf{1}} \cdot \nabla \phi_h = - \int_{U(T)} A e_i \cdot \nabla \phi_h \quad \mbox{for all } \phi_h \in \mathring{V}_h^r(U(T)).
\end{align}
For $\Phi_H \in V_H$ we define the corrector $Q_h^{\mathcal{U},{\mathbf{1}}}(\Phi_H) \in V_{h,\T_H}$ by
\begin{align*}
 Q_h^{\mathcal{U},{\mathbf{1}}}(\Phi_H) := \sum_{T \in \T_H} \chi_{T}(x) \sum_{i=1}^d \partial_{x_i} \Phi_H(x_T) \hspace{2pt} w_{h,T,i}^{\mathcal{U},{\mathbf{1}}}(x).
\end{align*}
Let $u_H^{\mathbf{1}} \in V_H$ be the solution of \eqref{coarse-scale-equation-in-framework-pg}, i.e., 
\begin{align*}
\sum_{T \in \T_H} \int_{T} A \left( \nabla u_H^{\mathbf{1}} + \nabla Q_h^{\mathcal{U},\alpha}(u_H^{\mathbf{1}}) \right) \cdot \nabla \Phi_H  = \int_{\Omega} f \Phi_H \quad \mbox{for all} \enspace \Phi_H \in V_H.
\end{align*}
Then, $u^{\mathbf{1},\mbox{\rm \tiny MsFEM}}_H:=u_H^{\mathbf{1}} + Q_h^{\mathcal{U},{\mathbf{1}}}(u_H^{\mathbf{1}})$ defines the MsFEM approximation obtained with Oversampling Strategy {\bf 1}. Obviously, $u^{\mathbf{1},\mbox{\rm \tiny MsFEM}}_H$ solves \eqref{classical-formulation-strategy-1}.
\end{osstrategy}
\begin{remark}
The explicit boundary condition for the local problems (\ref{strategy-hou-wu-local-problem}) is often missing in literature (c.f. \cite{Hou:Wu:1997,Efendiev:Hou:2009}). However, it seems that these computations were performed for the case described above, i.e. the solution $\tilde{\Phi}^T_i$ of (\ref{strategy-hou-wu-local-problem}) takes the values of an affine function on $\partial U(T)$ (c.f. \cite{Efendiev:Hou:Wu:2000,Hou:Wu:Zhang:2004} which also refer to the numerical experiments in \cite{Hou:Wu:1997}). In some works (c.f. \cite{Efendiev:Hou:Ginting:2004}) the local problems (\ref{strategy-hou-wu-local-problem}) are formulated with the boundary condition $\tilde{\Phi}^T_i=\Phi^T_i$ on $\partial U(T)$. This seems to be a mistake because the new basis functions will be equal to zero whenever $U(T)$ is larger than the support of the original basis functions.
\end{remark}

\subsection{Oversampling motivated from homogenization theory}
The second type of oversampling is motivated from numerical homogenization theory. Assume that we regard
\begin{align*}
\mbox{find} \enspace u^{\varepsilon} \in H^1_0(\Omega) \enspace \mbox{with} \quad \int_{\Omega} A^{\varepsilon} \nabla u^{\varepsilon} \cdot \nabla \Phi = \int_{\Omega} f \Phi \quad \mbox{for all } \Phi \in H^1_0(\Omega)
\end{align*}
and assume that $A^{\varepsilon}$ is uniformly bounded and coercive in $\varepsilon$, that $A^{\varepsilon}$ is $H$-convergent to some matrix $A^0$ and that $u^{\varepsilon} \rightharpoonup u^0$ in $H^1(\Omega)$, where $u^0 \in H^1_0(\Omega)$ is called the homogenized solution.

Then, a numerical approximation of the homogenized solution $u^0$ can be obtained by discretizing a more convenient equation (see (\ref{gloria-equation}) below). For this purpose, let $B(x,\eta)$ denote an open ball centered at $x\in \Omega$ with radius $\eta>0$ and let $N(x,\eta)$ denote an open neighborhood of $x \in \Omega$ with a Lipschitz boundary. It is assumed that there exist $0 < c \le C$ so that for all $\eta>0$ and all $x \in \Omega$ there holds $c|B(x, \eta)| \le |N(x, \eta)| \le C|B(x, \eta)|$. We seek $u^{\varepsilon,\eta,\zeta} \in H^1_0(\Omega)$ that solves
\begin{multline}
\label{gloria-equation} \int_{\Omega} |N(x,\eta)|^{-1}\int_{N(x, \eta)} A^{\varepsilon} (y)\left( \nabla u^{\varepsilon,\eta,\zeta}(x) + \nabla_y Q(u^{\varepsilon,\eta,\zeta})(x,y) \right) \cdot \nabla \Phi(x) \hspace{2pt} dy \hspace{2pt} dx\\=\int_{\Omega} f(x) \Phi(x) \hspace{2pt} dx
\end{multline}
for $\Phi \in H^1_0(\Omega)$, where for given $\Psi \in H^1_0(\Omega)$ the corrector $Q(\Psi)(x,\cdot)\in H^1_0(N(x, \eta+\zeta))$ is determined by 
\begin{align*}
\int_{N(x, \eta+\zeta)} A^{\varepsilon} (y)\left( \nabla \Psi(x) + \nabla_y Q(\Psi)(x,y) \right) \cdot \nabla \phi(y) \hspace{2pt} dy = 0 \text{ for all }\psi \in H^1_0(N(x, \eta+\zeta)).
\end{align*}
If $\zeta=\zeta(\eta)$ and $\lim_{\eta \rightarrow 0} \frac{\zeta(\eta)}{\eta} = 0$ then it holds 
\begin{align*}
\lim_{\eta,\zeta \rightarrow 0} \lim_{\varepsilon \rightarrow 0} \| u^0 - u^{\varepsilon,\eta,\zeta}\|_{H^1(\Omega)} = 0.
\end{align*}
As a consequence thereof, we get that
\begin{align*}
\lim_{\eta,\zeta \rightarrow 0} \lim_{\varepsilon \rightarrow 0} \| u^{\varepsilon} - u^{\varepsilon,\eta,\zeta}\|_{L^2(\Omega)} = 0.
\end{align*}
This result was shown by Gloria \cite{Gloria:2006,Gloria:2008} in a general nonlinear setting. Since $u^{\varepsilon,\eta,\zeta}$ yields a good approximation of $u^{\varepsilon}$, the result suggests to look at discretizations of (\ref{gloria-equation}). This was for instance exploited in \cite{Henning:Ohlberger:2011_2}. A general numerical framework that can be seen as a discretization of (\ref{gloria-equation}) was proposed in \cite{Henning:Ohlberger:Proceedings:2012}. Particularly, the HMM and the MsFEM are recovered from the framework, which leads to a straightforward oversampling strategy. This strategy can be formulated as follows (c.f. \cite{Gloria:2008,Henning:Ohlberger:Proceedings:2012,Henning:Ohlberger:Schweizer:2012}).
\begin{osstrategy}
For $i \in \{ 1,2,..,d\}$, let $w_{h,T,i}^{\mathcal{U},\mathbf{2}} \in \mathring{V}_h^r(U(T))$ solve
\begin{align}
\label{corrector-os2}\int_{U(T)} A \nabla w_{h,T,i}^{\mathcal{U},\mathbf{2}} \cdot \nabla \phi_h = - \int_{U(T)} A e_i \cdot \nabla \phi_h \quad \mbox{for all } \phi_h \in \mathring{V}_h(U(T)).
\end{align}
and for $\Phi_H \in V_H$, let $Q_h^{\mathcal{U},{\mathbf{2}}}(\Phi_H) \in V_{h,\T_H}$ denote the corrector given by (\ref{relation-corrector-and-basis}). If $u_H^{\mathbf{2}} \in V_H$ is the solution of \eqref{coarse-scale-equation-in-framework-pg} , then $u_H^{\mathbf{2}} + Q_h^{\mathcal{U},{\mathbf{2}}}(u_H^{\mathbf{2}})$ defines the MsFEM approximation obtained with Oversampling Strategy {2}.
We therefore denote
\begin{align*}
 u^{\mathbf{2},\mbox{\rm \tiny MsFEM}}_H := u_H^{\mathbf{2}} + Q_h^{\mathcal{U},{\mathbf{2}}}(u_H^{\mathbf{2}}).
\end{align*}
\end{osstrategy}
We immediately see, that Oversampling Strategy {1} and {2} only differ in the fine scale trial space for the local problems and that they are identical for $U(T)=T$, even though Strategy 2 was formulated independently of Strategy 1. In \cite{Hou:Wu:Zhang:2004,Efendiev:Hou:Wu:2000}, Strategy 2 is written in terms of an asymptotic expansion in the periodic case.
Also note that this second approach is closely related to the Heterogeneous Multiscale Finite Element Method, where the same type of oversampling is used (c.f. \cite{E:Engquist:2003,E:Engquist:2003-2,E:Engquist:2005,Henning:Ohlberger:2009}). Notably, HMM and MsFEM can be reinterpreted in a common homogenization framework (c.f. \cite{Gloria:2006,Gloria:2008}) and in a common numerical framework (c.f. \cite{Henning:Ohlberger:Proceedings:2012}).

\subsection{Discussion of the strategies}

As we just discussed, there are two widely used strategies for oversampling for the MsFEM. However, the difference between both approaches is only minor and the behavior of the resulting approximations appears to be qualitatively the same. The small difference in the local trial spaces does not seem to have a significant impact. At least, the error estimates available for strategies 1 and 2 are very similar. The literature does not even distinguish between these strategies. For instance, \cite{Chen:Savchuk:2007} (using Oversampling Strategy 1) claims to generalize the results of \cite{Efendiev:Hou:Wu:2000} (using Oversampling Strategy 2). Such a mixture of strategies can be observed in several works on this topic. To the best of our knowledge, even though both approaches seem to behave identical, a rigorous proof of this conjecture is still missing. Strategy 1 suggests to fix the corrector $Q_T^{\mathbf{1}}(\Phi)$ in the corners of the coarse grid element $T$ (forcing 
it to 
zero), whereas the corrector proposed by Strategy 2 does not have such a restriction leaving it completely free in these corners.

\begin{remark}
\label{gloria-comparison-pg-symmetric}
As already mentioned, the MsFEM might be also considered in a symmetric formulation (c.f. \cite{Efendiev:Hou:Wu:2000}), i.e. the coarse scale equation reads: find $u_H \in V_H$ with
\begin{align*}
&\sum_{T \in \T_H} \int_T A \left( \nabla u_H^{\alpha} + \nabla Q_h^{\mathcal{U},{\alpha}}(u_H^{\alpha}) \right) \cdot \left( \nabla \Phi_H + \nabla Q_h^{\mathcal{U},{\alpha}}(\Phi_H) \right) \\
&\qquad = \int_{\Omega} f (\Phi_H + Q_h^{\mathcal{U},{\alpha}}(\Phi_H))
\end{align*}
for all $\Phi_H \in V_H$ and where $Q_h^{\mathcal{U},{\alpha}}$ is defined either with Oversampling Strategy 1 or 2. However, the theoretical and numerical results in \cite{Hou:Wu:Zhang:2004} show that this version of the method still suffers from resonance errors. One explanation was suggested by Gloria \cite{Gloria:2008} who proposed a simple computation:
\begin{eqnarray*}
\lefteqn{\int_T A \left( \nabla u_H^{\alpha} + \nabla Q_h^{\mathcal{U},{\alpha}}(u_H^{\alpha}) \right) \cdot \left( \nabla \Phi_H + \nabla Q_h^{\mathcal{U},{\alpha}}(\Phi_H) \right)}\\
&=&\int_T A \left( \nabla u_H^{\alpha} + \nabla Q_h^{\mathcal{U},{\alpha}}(u_H^{\alpha}) \right) \cdot \nabla \Phi_H \\
&\enspace& \quad + \int_T A \left( \nabla u_H^{\alpha} + \nabla Q_h^{\mathcal{U},{\alpha}}(u_H^{\alpha}) \right) \cdot \nabla Q_h^{\mathcal{U},{\alpha}}(\Phi_H)\\
&=& \int_T A \left( \nabla u_H^{\alpha} + \nabla Q_h^{\mathcal{U},{\alpha}}(u_H^{\alpha}) \right) \cdot \nabla \Phi_H\\
&\enspace& \quad + \int_{U(T) \setminus T} \hspace{-14pt} A \left( \nabla u_H^{\alpha} + \nabla Q_h^{\mathcal{U},{\alpha}}(u_H^{\alpha}) \right) \cdot \nabla Q_h^{\mathcal{U},{\alpha}}(\Phi_H).
\end{eqnarray*}
This means that the effective MsFEM bilinear forms in Petrov-Galerkin and non-Petrov-Galerkin formulation differ in the term
\begin{align*}
\sum_{T\in\T_H} \int_{U(T) \setminus T} A \left( \nabla u_H^{\alpha} + \nabla Q_h^{\mathcal{U},{\alpha}}(u_H) \right) \cdot \nabla Q_h^{\mathcal{U},{\alpha}}(\Phi_H),
\end{align*}
which still seems to contain the problematic boundary layers that we tried to get rid off. Observe that we integrate over the layer $U(T) \setminus T$. This is exactly the region we encounter unpleasant boundary effects of the correctors $Q_h^{\mathcal{U},{\alpha}}(u_H)$ and $Q_h^{\mathcal{U},{\alpha}}(\Phi_H)$. This might imply that preference should be given to the Petrov-Galerkin formulation. Note, however, that uniqueness and existence of discrete solutions has not been proved for general oversampling so far. 
\end{remark}
Let us review the two essential results concerning the convergence of MsFEM approximations with oversampling. The first result is due to Gloria and the most general result currently available for Strategy 2.
\begin{theorem}
\label{gloria-theorem}
Let $f\in L^2(\Omega)$ and $A^{\varepsilon}\in L^\infty(\Omega,\mathbb{R}^{d\times d})$ be a sequence of (possibly non-symmetric) matrices with uniform spectral bounds $\gmin>0$ and $\gmax\geq\gmin$, 
\begin{equation}\label{e:spectralbound-Aeps}
\sigma(A^{\varepsilon}(x))\subset [\gmin,\gmax]\quad\text{for almost all }x\in \Omega \enspace \mbox{and for all } \varepsilon>0
\end{equation}
and assume that $A^{\varepsilon}$ is $H$-convergent. Let furthermore $u_H^{\varepsilon}\in V_H$ denote the corresponding MsFEM approximation obtained with Oversampling Strategy 2 and let 
\begin{align*}
\frac{\mbox{\rm diam}(U(T)) - \mbox{\rm diam}(T)}{\mbox{\rm diam}(T)}\rightarrow 0 \enspace \mbox{for} \enspace H \rightarrow 0.
\end{align*}
Then we have:
\begin{align*}
\lim_{H \rightarrow 0} \lim_{\varepsilon \rightarrow 0} \| u^{\varepsilon} - u_H^{\varepsilon} \|_{L^2(\Omega)} = 0.
\end{align*}
\end{theorem}
The proof for general non-symmetric coefficients is given in \cite[Theorem 5]{Gloria:2012} and the case of nonlinear problems is presented in \cite[Theorem 6 and Remark 7]{Gloria:2008}. At first glance the result appears counter-intuitive in the sense that  it suggests to let the oversampling converge to zero. However, the first limit is in $\varepsilon$, which makes the relative thickness $\frac{\varepsilon}{\dmin}$ of the oversampling layer grow to infinity. Hence, the correct interpretation is that for fixed $\varepsilon$ the computational domains should blow up to infinity. In this case, the optimal corrector problem is an equation formulated on whole $\mathbb{R}^d$. These corrector problems are exactly the cell problems known from periodic and stochastic homogenization theory. In the periodic setting the classical cell problem can be extended to the $\mathbb{R}^d$ by periodicity and in the stochastic setting they are directly formulated in $\mathbb{R}^d$ to obtain the correct stochastic average (c.f. \cite{Jikov:Kozlov:Oleinik:1994}).

Theorem \ref{gloria-theorem} gives a clear message in the case of extremely small microscopic variations. If $\varepsilon$ (the characteristic length scale of the fine scale oscillations) is (globally) sufficiently small then the resulting MsFEM approximation yields very good approximations. This is a very important result, but it is purely qualitative. E.g., it does not answer the question how (thick) to choose an oversampling patch. We cannot predict how the method behaves if there is a large spectrum of oscillations without a scale separation. For instance, we might encounter variations, where it is hard to tell which of them are macroscopic and which are microscopic (i.e. '$\varepsilon$-dependent'). In practice, we do not construct an artificial sequence in $\varepsilon$, we only have a given scenario and a given set of data.

$\\$
The next theorem due to Hou, Wu and Zhang is much more restrictive, but it gives a more quantitive answer than Theorem \ref{gloria-theorem}:
\begin{theorem}
\label{Hou:Wu:Zhang:2004-theorem}
Assume that $d=2$, $f\in L^2(\Omega)$ and that $A$ is a bounded, elliptic, symmetric and $\varepsilon$-periodic $C^3$-matrix, i.e. $A(x)=A_p(\frac{x}{\varepsilon})$, with $A_p \in C^3( [0,1]^d ,\mathbb{R}^{d\times d}_{sym})$ being periodic. Let $u_H^{\varepsilon}\in V_{h,\T_H}$ denote the MsFEM approximation obtained with Oversampling Strategy 2. Then:
\begin{align*}
 \|u_H^{\varepsilon} - u^{\varepsilon} \|_{L^2(\Omega)} &\le C \left( \frac{\varepsilon}{\dmin} + H + \varepsilon (\log H)^{\frac{1}{2}} \right),\\
 \left( \sum_{T \in \T_H} \|\nabla u_H^{\varepsilon} - \nabla u^{\varepsilon}\|_{L^2(\Omega)}^2 \right)^{\frac{1}{2}} &\le C \left( \frac{\varepsilon}{\dmin} + H + \varepsilon^{\frac{1}{2}} \right).
\end{align*}
\end{theorem}
A proof of this theorem is given in \cite{Hou:Wu:Zhang:2004}. The assumption $d=2$ seems to be essential for their strategy. Note that in \cite{Hou:Wu:Zhang:2004} the theorem is formulated without the $\frac{1}{\dmin}$ contribution. Instead, the authors make the assumption that the oversampling layer is sufficiently large. Following their proofs one can easily see that the generalized estimate reads as above (c.f. \cite{Efendiev:Hou:2009} for the case $\dmin=CH$). In particular, the $\frac{\varepsilon}{\dmin}$-term describes the decay of the error between the exact corrector and the corrector with wrong boundary condition in a coarse element $T$. The decay turns out to be inverse proportional to the thickness of the layer. Because of the $\varepsilon$ scaling of the solution, the effective term becomes $\frac{\varepsilon}{\dmin}$. This seems to be a sharp estimate for the decay due to the findings in \cite{Efendiev:Hou:Wu:2000} and \cite{Chen:Hou:2003}. A proof of Theorem \ref{Hou:Wu:Zhang:2004-theorem} for 
Oversampling 
Strategy 1 can be achieved in the same fashion as in \cite{Hou:Wu:Zhang:2004}. Theorem \ref{Hou:Wu:Zhang:2004-theorem} predicts the following: if locally  O$(H)$=O$(\varepsilon)$ the patch size of the local problems must not be of order O$(H)$ to preserve convergence. Still, the theorem only gives an answer of how to choose the oversampling patches $\mathcal{U}$ if $\varepsilon$ is a known parameter.

If the thickness of the oversampling layer is of order O$(h)$ both estimates in Theorem \ref{Hou:Wu:Zhang:2004-theorem} receive an order O$(\frac{\varepsilon}{h})$ term and the right hand sides remain large. In general, the thickness $\dmin$ must be large in comparison to $\varepsilon$. Analytically, this implies that O$(H)$-oversampling might be not enough in regions where we deal with resonance errors due to O$(H)$=O$(\varepsilon)$.
This seems to show up in the numerical experiments in \cite{Hou:Wu:Zhang:2004}, where the authors observe a stagnation in the convergence of the $H^1$-error for $H$ entering the region with O$(H)$=O$(\varepsilon)$. The effect on the $L^2$-error is less strong. However, the value of $\dmin$ is missing in the experiments in \cite{Hou:Wu:Zhang:2004} so we can only assume that $\dmin$ is of order $H$. Otherwise the computation of the MsFEM basis functions becomes quite expensive. However, we note that there also exists a modification of Oversampling Strategy 2 proposed by Gloria (c.f. \cite[Theorem 3.1]{Gloria:2011} and \cite[Paragraphs 5.3 and 5.4]{Gloria:2012}) where the local problems are regularized by adding the term $\kappa^{-1} (w_{h,T,i}^{\mathcal{U},\mathbf{2}},\phi_h)_{L^2(U(T))}$ (for large $\kappa>0$) to the left hand side of problem (\ref{corrector-os2}). Using this modified strategy, the Theorem-\ref{Hou:Wu:Zhang:2004-theorem}-type estimates can be improved enormously, even without restrictions on 
space dimension and much weaker assumptions on the regularity of $A^{\varepsilon}$.

\begin{remark}
In \cite{Efendiev:Hou:Wu:2000}, the symmetric version (\ref{symmetric-version-msfem}) of the MsFEM is considered. Here, the derived $L^2$-estimate reads: 
\begin{align*}
 \|u_H^{\varepsilon} - u^{\varepsilon} \|_{L^2(\Omega)} &\le C \left( \varepsilon + H^2 + \varepsilon (\log H) + \frac{\varepsilon}{\dmin} + C_r \left(\frac{\varepsilon}{H}\right)^2 \right),
\end{align*}
where $u_H^{\varepsilon}$ denotes the MsFEM approximation of the symmetric problem (\ref{symmetric-version-msfem}) determined with Oversampling Strategy 1. Due to the numerical experiments in \cite{Efendiev:Hou:Wu:2000}, $C_r$ seems to have a considerable size, so that this term is dominating the estimate if locally O$(\varepsilon)=$O$(H)$. We notice that the estimate is worse then the $L^2$-estimate for the Petrov-Galerkin version of the method, because the last term cannot be reduced even for large $\dmin$. These observations are consistent with Remark \ref{gloria-comparison-pg-symmetric}. However, this leads to an additional problem of the MsFEM with Strategy 1 or Strategy 2. On the one hand, the Petrov-Galerkin version should be preferred over the symmetric version (see the estimates). On the other hand, the existence and uniqueness of the corresponding MsFEM approximations has not been established so far, not to mention stability. For the symmetric version, we can simply exploit the ellipticity of $A$ 
to 
conclude that 
the method is well posed and stable. For the Petrov-Galerkin version there is no such argument. The only result is a perturbation result due to Gloria \cite{Gloria:2008}, saying that if the oversampling size is small enough (i.e. if the difference between Petrov-Galerkin formulation and symmetric formulation is small enough) then we still have existence and uniqueness. The lack of knowledge regarding the general well-posedness of the Petrov-Galerkin MsFEM with Strategy 1 and 2 is a big issue of these approaches.
\end{remark}
The $\varepsilon$-terms in the estimates that cannot be reduced with $H\gtrsim\varepsilon$ should be seen as fixed modeling errors. They describe the error between exact solution and homogenized solution. In a general non-periodic non-stochastic scenario they cannot be quantified.

In conclusion we have two findings. First, in general, both approaches do not show clear asymptotics for a convergence to the exact solution (for $H\gtrsim\varepsilon$). There is always a remainder of order $\varepsilon$, even if $U(T)=\Omega$. In particular, this is a problem if $\varepsilon$ unknown or if the micro structure is heterogeneous. Second, if the modeling error of order $\varepsilon$ is negligible, Theorem \ref{Hou:Wu:Zhang:2004-theorem} still suggests that linear convergence (with respect to $H$) can only be achieved  if the oversampling thickness scales with O$(1)$ which makes the local problems prohibitively expensive. Since the estimate for the decay rate $\frac{\varepsilon}{\dmin}$ of the corrector error is sharp (in the periodic setting) we cannot hope for much improvement of the final error estimates stated in Theorem \ref{Hou:Wu:Zhang:2004-theorem}.

$\\$
We may summarize the following issues that we address to solve with our new oversampling strategy to be proposed in the next section:
\begin{itemize}
 \item[a)] elimination of resonance errors of any kind,
 \item[b)] clear prediction for the size of oversampling patches without explicit knowledge about the micro-structure or scale separation,
 \item[c)] construction of a conforming approximation in $H^1_0(\Omega)$,
 \item[d)] a quantitative error analysis in $H$ without restrictive regularity assumptions on the coefficients and for all space dimensions,
 \item[e)] a-priori error estimation in the fully discrete setting (previous results were obtained under the assumption that the local problems are solved exactly),
 \item[f)] formulation of a stable approach for which we can guarantee existence and uniqueness of the resulting MsFEM approximation,
 \item[g)] prevention of unstable splittings due to point evaluations as, e.g., required to implement the constraint in oversampling strategy 1 (cf. Definition of $\mathring{V}_h^r(U(T))$ in \eqref{def-reduced-space-stragey-1}).
\end{itemize}
Note that points d) and e) could be done in the periodic setting for Strategy 1 and 2 by using e.g. the techniques presented in \cite{Gloria:2011,Gloria:2012}.
\section{Constrained oversampling}
\label{section-new-oversampling-strategy}

In this section we introduce a third oversampling strategy for which we derive a quantitative a-priori error estimate. The results are presented in Subsection \ref{subsection-new-strategy} and a corresponding proof is given in Subsection \ref{subsection-proof-main-result}. All the results require solely the assumptions stated in Section \ref{subsection-setting} to be satisfied, i.e., $A\in L^\infty(\mathbb{R}^{d\times d}_{\operatorname{sym}})$  uniformly positive definite and $f\in L^2(\Omega)$.

\subsection{New strategy and quantitative error estimates}
\label{subsection-new-strategy}
In the following, let $\mathcal{N}$ denote the set of Lagrange points of the coarse grid $\mathcal{T}_H$. For a given node $z \in \mathcal{N}$, $\Phi_z \in V_H$ denotes the corresponding nodal basis function as before.

Our new approach is based on some multiscale decomposition of the space $V_h$,
\begin{equation}\label{e:decomposition}
V_h = V_H \oplus \Vf,
\end{equation}
where the space $\Vf$ contains the 'fine scale' functions of $V_h$, i.e., functions that are not captured by $V_H$. More precisely, we choose $\Vf$ to be the kernel of some Cl\'ement-type quasi-interpolation operator $\Ic : H^1_0(\Omega) \rightarrow V_H$, 
\begin{align}\label{e:Vf}
\Vf:= \{ v \in V_h\;|\;\enspace \Ic(v)=0\}.
\end{align}
Several choices for $\Ic$ are possible. We refer to \cite{2011arXiv1110.0692M} for an axiomatic characterization. In this paper, for the sake of simplicity, we we choose the particular operator introduced in \cite{MR1736895}. Given $v\in H^1_0(\Omega)$, $\Ic v := \sum_{z\in\mathcal{N}}(\Ic v)(z)\Phi_z$ is determined by the nodal values 
\begin{equation}\label{e:clement}
 (\Ic v)(z):=\frac{\int_{\Omega} v \Phi_z}{\int_{\Omega} \Phi_z}\quad \text{ for }z\in\mathcal{N}.
\end{equation}
The nodal values are weighted averages of the function over nodal patches $\omega_{z}:=\operatorname{supp}\Phi_z$.
The operator is linear, surjective, bounded, and invertible on the finite element space $V_H$. Hence, the decomposition  \eqref{e:decomposition} exists and is stable; it is even orthogonal in $L^2(\Omega)$. 

Recall the (local) approximation and stability properties of the interpolation operators $\Ic$ \cite{MR1736895}: There exists a generic  constant $C$ such that for all $v\in H^1_0(\Omega)$ and for all $K\in\T_H$ it holds
\begin{equation}\label{e:interr}
 H_T^{-1}\|v-\Ic v\|_{L^{2}(K)}+\|\nabla(v-\Ic v)\|_{L^{2}(K)}\leq C \| \nabla v\|_{L^2(\omega_K)},
\end{equation}
where $\omega_K:=\cup\{K'\in\T_H\;\vert\;K'\cap K\neq\emptyset\}$. The constant $C$ depends on the shape regularity of the finite element mesh $\T_H$ but not on the local mesh size $H_T:=\operatorname{diam}(T)$.

\begin{remark}[Nodal interpolation]
Since we consider a fully discrete setting, where corrector problems are solved in the fine scale finite element space $V_h$, we could have chosen nodal interpolation instead of Cl\'ement-type interpolation. The subsequent definitions and results will be almost verbatim the same. However, nodal interpolation does not satisfy the estimate \eqref{e:interr} with an $h$-independent constant if $d>1$. The best constant $C=C_d(h)$ reads $C_2(h)=\log(H/h)$ and $C_3(h)=(H/h)^{-1}$  depending on the spatial dimension $d$ (c.f. \cite{Yserentant:1986}). Since this constant enters basically all error estimates below, we would end up with an $h$-dependence of the multiplicative constants in the final error estimates. In 2d this can  still be acceptable because the dependence on $h$ is only logarithmic. 
\end{remark}

With the decomposition \eqref{e:decomposition} we do not search the local correctors in the full fine scale space $V_h$, but only in the constrained space $\Vf$. The advantage is the following: as stated in the previous section for oversampling strategy 1 and 2, the standard decay for the difference between the local correctors and the global 'exact' corrector is of order $\frac{1}{\dmin}$ (see Theorem \ref{Hou:Wu:Zhang:2004-theorem}), but in the constrained space $\Vf$ we can achieve an exponential-type decay (cf. Lemma~\ref{l:decay} below).

We now propose our new Oversampling Strategy:
\begin{osstrategy}[Constrained oversampling]
Let $\Vf$ denote the space given by \eqref{e:Vf} and define 
\begin{equation}
\label{localspacedefos3} \oVf(U(T)):=\{v_h\in \Vf\;\vert\;
v_h\vert_{\Omega\setminus U(T)}=0\}.
 \end{equation}
The local correctors $w_{h,T,i}^{\mathcal{U},\mathbf{3}} \in \oVf(U(T))$ (for $i \in \{ 1,2,..,d\}$) are defined as the (unique) solutions of 
\begin{align}\label{e:correctorlocal}
\int_{U(T)} A \nabla w_{h,T,i}^{\mathcal{U},\mathbf{3}} \cdot \nabla \phi_h = - \int_{T} A e_i \cdot \nabla \phi_h
 \quad \mbox{for all } \phi_h \in \oVf(U(T)).
\end{align}
For general $\Phi_H \in V_H$ we define the correction operator $Q_h^{\mathcal{U},{\mathbf{3}}}: V_H \rightarrow V_h$ by
\begin{align*}
Q_h^{\mathcal{U},{\mathbf{3}}}(\Phi_H)(x):= \sum_{T \in \T_H} \sum_{i=1}^d \partial_{x_i}\Phi_H(x_T) w_{h,T,i}^{\mathcal{U},\mathbf{3}}(x).
\end{align*}
The global coarse scale approximation $u_H^{\mathbf{3}} \in V_H$ is the solution of (\ref{symmetric-version-msfem}), i.e. it solves
\begin{eqnarray}
\label{e:MsVMMsym}\nonumber\lefteqn{\mathcal{A}^{\mathbf{3}}(u_H^{\mathbf{3}},\Phi_H) :=\int_{\Omega} A \left( \nabla u_H^{\mathbf{3}} + \nabla Q_h^{\mathcal{U},{\mathbf{3}}}(u_H^{\mathbf{3}}) \right) \cdot \left( \nabla \Phi_H + \nabla Q_h^{\mathcal{U},{\mathbf{3}}}(\Phi_H) \right)}\\
&=& \int_{\Omega} f (\Phi_H + Q_h^{\mathcal{U},{\mathbf{3}}}(\Phi_H)) \quad\text{for all }\Phi_H \in V_H. 
\end{eqnarray}
The corresponding MsFEM approximation is given by
\begin{align*}
 u^{\mathbf{3},\mbox{\rm \tiny MsFEM}}_H := u_H^{\mathbf{3}} + Q_h^{\mathcal{U},{\mathbf{3}}}(u_H^{\mathbf{3}}).
\end{align*}
\end{osstrategy}

Using the above definition of the localized space $\oVf(U(T))$ does not assure that our new method boils down to the classical MsFEM in the case without oversampling. Nevertheless, this can be achieved by introducing a localized interpolation operator. Given some element $T\in\T_H$ and an admissible patch $U(T)$, we can define $\Icloc{U(T)}$ to be the Cl\'ement-type quasi-interpolation operator with respect to the domain $U(T)$ (with extension by zero in $\Omega\setminus U(T)$). Then, the localized space $\oVf(U(T))$ can be defined in analogy to $\Vf$ with $\Icloc{U(T)}$ replacing $\Ic$. With this modification we obtain the classical MsFEM for $U(T)=T$. This is only a subtle detail and all results still remain valid for these modified local spaces, however, this version would generate some technicalities in the proofs later on, which is why we decided to work with the definition (\ref{localspacedefos3}).

\begin{remark}
The crucial differences between the classical Strategies 1, 2 and Oversampling Strategy 3 are the following:
\begin{itemize}
 \item[a)] The variational problem for the local corrector in Oversampling Strategy 3 is posed in the constrained space $\mathring{W}_h(U(T))$ whereas the classical corrector problem seeks the local corrector in the full space $\mathring{V}_h(U(T))$ restricted to the patch.
 \item[b)] The support of the integrals on the right hand side in (\ref{corrector-os1}) and (\ref{corrector-os2}) is $U(T)$. In our new version we use only the element $T$. This allows us to exploit nice summation properties of the local projectors, without using indicator functions $\chi_T$ that lead to discontinuities.
 \item[c)] In the classical setting, the local correctors are restricted to the corresponding elements to derive the global corrector. For Oversampling Strategy 3, we simply sum up (weighted by the coefficients of the finite element function) the local contributions to get the global corrector. Note that our global corrector is conforming in the sense that its image is a subset of $V_h\subset H^1_0(\Omega)$ whereas the classical setup leads to a non-conforming corrector.
 \item[d)] In Strategy 3, we do not use a Petrov-Galerkin formulation for the global problem (\ref{e:MsVMMsym}). Since $A$ is assumed symmetric, a symmetric discretization appears more natural. Furthermore, we immediately inherit coercivity for the global bilinear form $\mathcal{A}^{\mathbf{3}}$. This gives us existence and uniqueness of $u_H^{\mathbf{3}}$ and the arising MsFEM approximation is well posed and the method stable. The typical disadvantage of the symmetric version which still suffers from resonance errors (which is why Petrov-Galerkin is typically preferred) does not remain for our strategy.
 \item[e)] In contrast to Oversampling Strategy 1 and 2, the corrector $Q_h^{\mathcal{U},{\mathbf{3}}}(\Phi_H)$ does not preserve the support of $\Phi_H$. In other words, the set of multiscale basis functions $\Phi_z+ Q_h^{\mathcal{U},{\mathbf{3}}}(\Phi_z)$ with $z \in \mathcal{N}$ has an extended support. This results in a loss of sparsity in the stiffness matrix that corresponds with the global problem (\ref{e:MsVMMsym}). In order to still assemble the stiffness matrix in an efficient way, one might store the intersection domain for each given two oversampling patches (in storage types with low memory requirements). This can be easily done at the same time when the grids for the local patches are generated. Once all intersection domains are available, the matrix can be efficiently assembled. A quadrature rule that resolves the micro-structure is needed for each of the strategies.
\end{itemize}
\end{remark}

\begin{remark}[Perturbation of the right hand side]
We might also replace the right hand side of \eqref{e:MsVMMsym} by the term $\int_{\Omega} f \Phi_H$.
This introduces only a perturbation of order $\|H f\|_{L^2(\Omega)}$ in the $H^1$-error.
\end{remark}

\begin{remark}[Non-Symmetric formulation]
As for the classical strategies, one might also consider the non-symmetric Petrov-Galerkin formulations: find $u_H^{\mathbf{3}} \in V_H$ such that
\begin{align*}
\int_{\Omega} A \left( \nabla (u_H^{\mathbf{3}} + Q_h^{\mathcal{U},{\mathbf{3}}}(u_H^{\mathbf{3}})) \right) \cdot \nabla \Phi_H = \int_{\Omega} f \Phi_H \quad \mbox{for all } \Phi_H \in V_H.
\end{align*}
or
\begin{align*}
\int_{\Omega} A \left( \nabla u_H^{\mathbf{3}}  \right) \cdot \nabla (\Phi_H+Q_h^{\mathcal{U},{\mathbf{3}}}(\Phi_H)) = \int_{\Omega} f \Phi_H \quad \mbox{for all } \Phi_H \in V_H.
\end{align*}
\end{remark}

In the spirit of homogenization theory, one might pose the question if Theorem \ref{gloria-theorem} still holds for MsFEM approximations obtained with Oversampling Strategy 3. At least, this seems to be likely. The reason is that Theorem \ref{gloria-theorem} in particular covers the case without oversampling (see also \cite{Gloria:2006}) and the proof given in \cite{Gloria:2008} goes back to the arguments used for the case without oversampling. But for $\mathcal{U}=\T_H$ (no oversampling), strategies 1 and 2 are identical, and strategy 3 is at least close to the classical approach.
Especially concerning Strategy 3: if the thickness of the oversampling layer decreases faster than the coarse mesh size, we are almost in the case of Oversampling Strategy 1, up to a small perturbation of the source term that is of order $\max_{T\in \T_H} \frac{|U(T)\setminus T|}{|T|}$ and that converges to zero under the assumptions of Theorem \ref{gloria-theorem}. However, such arguments still need a detailed investigation. In this sense, one might carefully study if Strategy 3 also covers the 
homogenization setting established by Gloria, with $u_H^{\mathbf{3}}$ converging to the homogenized solution as in Theorem \ref{gloria-theorem}. This might be an interesting result to ensure that Strategy 3 is not worse than the classical strategies with respect to a homogenization setting.

$\\$
Besides the advantages of our new strategy mentioned previously, e.g., its conformity, stability, unique solvability, we formulate the main error estimate which is proved in  Subsection \ref{subsection-proof-main-result}.

\begin{theorem}[Quantitative a-priori error estimates]
\label{main-theorem}Assume that $A\in L^\infty(\Omega,\mathbb{R}^{d\times d}_{sym})$ and $f \in L^2(\Omega)$ as in the general assumptions in Section \ref{subsection-setting}. Let $\T_H$ be given coarse triangulation and let $\mathcal{U}$ denote a corresponding set of admissible patches, with the property $\dmin \gtrsim H \log(H^{-1})$. By $\T_h$ we denote a sufficiently accurate fine triangulation of $\Omega$ and by $u_h$ the associated finite element solution of (\ref{e:modelref}). If $u^{\mathbf{3},\mbox{\rm \tiny MsFEM}}_H$ is the MsFEM approximation determined with Oversampling Strategy 3 and if $u_H^{\mathbf{3}}$ denotes the corresponding coarse part, then the following a-priori error estimates holds true for arbitrary mesh sizes $H\geq h$:
\begin{align*}
 \| \nabla u_h - \nabla u^{\mathbf{3},\mbox{\rm \tiny MsFEM}}_H\|_{L^2(\Omega)} &\le C H,\\
 \| u_h - u^{\mathbf{3},\mbox{\rm \tiny MsFEM}}_H\|_{L^2(\Omega)} &\le C H^2,\\
 \| u_h - u_H^{\mathbf{3}}\|_{L^2(\Omega)} &\le C H.
\end{align*}
Here, $C$ denotes a generic constant that depends on $f$, $\gmin$ and $\gmax$ but {\rm not} on $H$, $h$, the regularity of the exact solution or the variations of $A$. Details on the constants are given in Theorems \ref{t:H1} and \ref{t:L2}.
\end{theorem}

\subsection{Proof of the main result}
\label{subsection-proof-main-result}

Before we prove the error estimates for the MsFEM with the correctors presented in Oversampling Strategy 3, we introduce some simplifying notations for this subsection.

\begin{definition}[Notations for Oversampling Strategy 3]
\label{notations-strat-3}
Let $w_{h,T,i}^{\mathcal{U},\mathbf{3}} \in \oVf(U(T))$ denote the local corrector basis given by (\ref{e:correctorlocal}), let $Q_h^{\mathcal{U},{\mathbf{3}}}$ denote the corresponding corrector operator from Strategy 3 and let $u_H^{\mathbf{3}}$ denote the arising (coarse) MsFEM approximation. In the following, we skip the redundant indices and use the following notation:
\begin{align*}
 w_T^i := w_{h,T,i}^{\mathcal{U},\mathbf{3}}, \quad \Qh:=Q_h^{\mathcal{U},{\mathbf{3}}}, \quad u_H:= u_H^{\mathbf{3}} \quad \mbox{and} \quad u^{\mbox{\tiny MsFEM}}:=u_H + \Qh(u_H).
\end{align*}
\end{definition}
The first lemma treats the (unpractical) case of maximal oversampling.
\begin{lemma}[Error estimate for maximal oversampling]\label{l:maximalos}
Let $U(T)=\Omega$ for all $T\in\T_H$. Then the multiscale approximation $u_H$ that solves \eqref{e:MsVMMsym} satisfies the error estimate
\begin{equation*}
 \|\nabla u_h-\nabla(u_H+\Qh(u_H))\|_{L^2(\Omega)}\lesssim \gmin^{-1}\|H f\|_{L^2(\Omega)},
\end{equation*}
where $u_h$ solves the reference problem \eqref{e:modelref}.

If, moreover, $(f,w_h)_{L^2(\Omega)}=0$ for all fine scale functions $w_h\in \Vf$, then $u_H+\Qh(u_H)=u_h$.
\end{lemma}
\begin{proof}
For $U(T)=\Omega$, $\Qh$ maps onto the fine scale space $W_h$. Given $\Phi_H\in V_H$, it is easily checked that $\Qh(\Phi_H) = \sum_{T \in \T_H} \sum_{i=1}^d \partial_{x_i}\Phi_H(x_T) w_T^i$ satisfies
\begin{align*}
a(\Qh(\Phi_H) , \phi_h) &= \int_{\Omega} A \nabla \left( \sum_{T \in \T_H} \sum_{i=1}^d \partial_{x_i}\Phi_H(x_T) w_T^i(x) \right) \cdot \nabla \phi_h \\
&= \int_{\Omega} A \left( \sum_{T \in \T_H} \sum_{i=1}^d \partial_{x_i}\Phi_H(x_T) \nabla w_T^i(x) \right) \cdot \nabla \phi_h \\
&= \sum_{T \in \T_H} \sum_{i=1}^d \partial_{x_i}\Phi_H(x_T) \int_{\Omega} A \nabla w_T^i(x) \cdot \nabla \phi_h \\
&= - \sum_{T \in \T_H} \sum_{i=1}^d \partial_{x_i}\Phi_H(x_T) \int_{T} A e_i \cdot \nabla \phi_h \\
&= - \sum_{T \in \T_H} \int_{T} A \nabla \Phi_H \cdot \nabla \phi_h \\
&= - a(\Phi_H ,\phi_h)
\end{align*}
for all $\phi_h \in \Vf$. This means that $\Qh$ is the orthogonal projection of $\Phi_H$ onto the fine scale space $\Vf$ with respect to the scalar product $a( \cdot , \cdot )$. This yields the orthogonal decomposition
\begin{align}\label{e:odecompostion}
V_h = \tilde{V}_H \oplus_{\perp_a} W_h , \;\mbox{where} \enspace \tilde{V}_H:= \{ \Phi_H + \Qh(\Phi_H)\;|\; \Phi_H \in V_H  \}.
\end{align}
Moreover, it holds Galerkin orthogonality, i.e., for $e_h:=u_h-(u_H+\Qh(u_H))$ and for arbitrary $\Phi_H + \Qh(\Phi_H) \in \tilde{V}_H$,
\begin{align}\label{e:galerkinortho}
\nonumber a( e_h, \Phi_H + \Qh(\Phi_H))&=a( u_h, \Phi_H + \Qh(\Phi_H)) - a( u_H+\Qh(u_H), \Phi_H + \Qh(\Phi_H)) \\
&\overset{\eqref{e:MsVMMsym}}{=} 0.
\end{align}
The combination of \eqref{e:odecompostion} and \eqref{e:galerkinortho} shows that $e_h \in \Vf$ and therefore $\Ic(e_h)=0$. We obtain
\begin{align*}
\gmin\|\nabla e_h \|_{L^2(\Omega)}^2 &\leq a(e_h,e_h) = a(u_h,e_h)= \int_{\Omega} f e_h = \int_{\Omega} f (e_h-\Ic(e_h)).
\end{align*}
The application of the Cauchy-Schwarz inequality on the element level and the estimate \eqref{e:interr} for the interpolation error yield the  assertion.
\end{proof}

\begin{corollary}
The new MsFEM is exact (up to the discretization error on the fine scale and oscillations $\|H f\|_{L^2(\Omega)}$ of the right-hand side $f$) in the limit of maximal oversampling. This results holds true independent of the upper spectral bound $\gmax$ and the variations of $A$. This is the next difference to the previous Strategies 1 and 2.
\end{corollary}

Although the error estimate in Lemma~\ref{l:maximalos} is encouraging, maximal oversampling is not feasible. We shall 
study the decay of the correctors away from element they are associated with.
For all $T\in\T_H $, define element patches in the coarse mesh $\T_H$ by
  \begin{equation}\label{def-patch-U-k}
    \begin{aligned}
      U_0(T) & := T, \\
      U_k(T) & := \cup\{T'\in \T_H\;\vert\; T'\cap U_{k-1}(T)\neq\emptyset\}\quad k=1,2,\ldots .
    \end{aligned}
  \end{equation}

\begin{lemma}[Decay of the ideal correctors]\label{l:decay}
Let $U(T)=\Omega$ for all $T\in\T_H$ in Oversampling Strategy 3
and let $w^i_T$ denote the corresponding local correctors defined in
Definition \ref{notations-strat-3} (and equation (\ref{e:correctorlocal})). Then, for all $T\in\T_H$ and all $k\in\mathbb{N}$,
$$\|A^{1/2}\nabla w^i_T\|_{L^2(\Omega\setminus U_k(T))}\lesssim e^{-r k}\|A^{1/2}\nabla w^i_T\|_{L^2(\Omega)},$$
where $r$ is a positive constant that depends on the square root of the contrast but not on the mesh size or the variations of $A$.
\end{lemma}

The proof of Lemma~\ref{l:decay} requires the definition of cut-off functions and an additional lemma. For $T\in\T_H$ and $\ell,k\in\mathbb{N}$ with $k > \ell$, define $\eta_{T,k,\ell}\in P_1(\T_H)$ with nodal values
\begin{equation}\label{e:cutoffH}
\begin{aligned}
 \eta_{T,k,\ell}(z) &= 0\quad\text{for all }z\in\mathcal{N}\cap U_{k-\ell}(T),\\
 \eta_{T,k,\ell}(z) &= 1\quad\text{for all }z\in\mathcal{N}\cap \left(\Omega\setminus U_k(T)\right),\text{ and}\\
 \eta_{T,k,\ell} (z)&= \frac{m}{\ell}\quad\text{for all }x\in\mathcal{N}\cap \partial U_{k-\ell+m}(T),\;m=0,1,2,\ldots,\ell.
\end{aligned}
\end{equation}
For a sketch in $1d$, see Figure \ref{eta-figure}.

\begin{figure}[t]
\centering
\includegraphics[width=0.40\textwidth,height=0.40\textwidth]{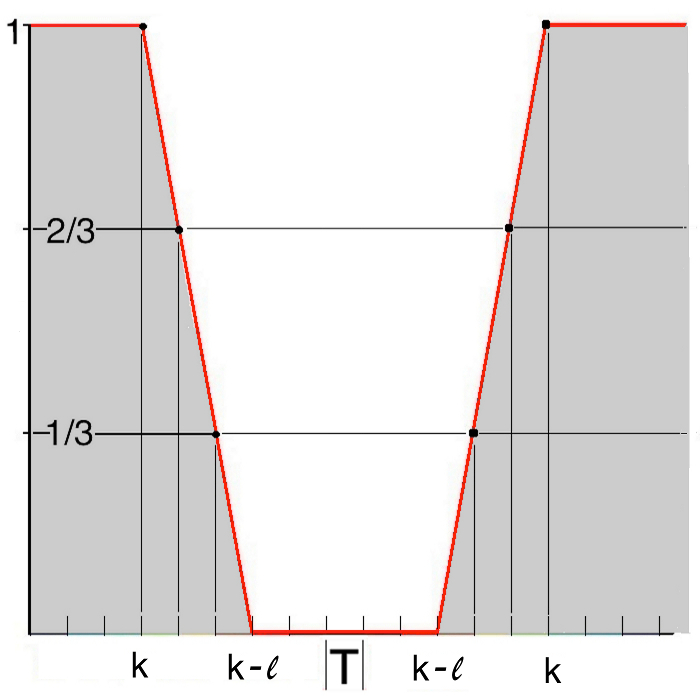}
\caption{Sketch of $\eta_{T,k,\ell}$ (red curve) in $1d$ for $k=5$ and $\ell=3$. $\eta_{T,k,\ell}$ is equal to zero on $T$ and also on the first 2 ($=k-\ell$) coarse grid layers around $T$. Then in grows linearly on the layers $\ell=3$ till $k=5$. On the remaining layers $\eta_{T,k,\ell}$ is constantly to $1$.}
\label{eta-figure}
\end{figure}

Given some $w\in \Vf$, the product $\eta_{T,k,\ell}w$ is not in $\Vf$ in general. However, the distance of $\eta_{T,k,\ell}w$ and $\Vf$ is small in the following sense. 
\begin{lemma}\label{l:cutoff}
Given $w\in \Vf$ and some cutoff function $\eta_{T,k,\ell}\in P_1(\T_H)$ as in \eqref{e:cutoffH}, there exists some $\tilde{w}\in \mathring{W}_h(\Omega \setminus U_{k-\ell-1}(T)) \subset W_h$  such that
\begin{equation*}
\|\nabla(\eta_{T,k,\ell}w-\tilde{w})\|_{L^2(\Omega)}\lesssim  \ell^{-1} \|\nabla w\|_{L^2(U_{k+2}(T)\setminus U_{k-\ell-2}(T))}.
\end{equation*}
\end{lemma}
\begin{proof}
Fix some $T\in\T_H$ and $k\in\mathbb{N}$ and let $\eta_\ell:=\eta_{T,k,\ell}$. The operator $I_h:H^1_0(\Omega)\cap C(\bar{\Omega})\rightarrow V^h$ denotes the nodal interpolant with respect to the mesh $\T_h$. Recall that for all quadratic polynomials $p$ and all $t\in\T_h$, $I_h$ fulfills the (local) approximation and stability estimates
\begin{equation}\label{e:intesth}
 \|\nabla (p-I_h p)\|_{L^2(t)}\lesssim h_t\|\nabla^2 p\|_{L^2(t)}\quad\text{and}\quad \|\nabla(I_h p)\|_{L^2(t)}\lesssim \| \nabla p\|_{L^2(t)}.
\end{equation} 
We will use this estimate for the $\T_h$-piecewise quadratic function $p=\eta_\ell w$. Since $\nabla^2\eta_\ell=\nabla^2 w=0$ in every $t\in\T_h$, we have that $\nabla^2_h \eta_\ell w = \nabla\eta_\ell\cdot\nabla w$ in $t$.

According to \cite[Lemma 1]{2011arXiv1110.0692M}, there exists some $v\in V^h$ such that 
\begin{equation}\label{e:lackproj}
\Ic v=\Ic I_h(\eta_\ell w),\;\|\nabla v\|_{L^2(\Omega)}\lesssim \|\nabla \Ic I_h(\eta_\ell w)\|_{L^2(\Omega)},\;\text{and }\operatorname{supp} (v)\subset \Omega \setminus U_{k-\ell-1}(T).  
\end{equation}
Hence, $\tilde{w}:=I_h(\eta_\ell w)-v\in\oVf(\Omega \setminus U_{k-\ell-1}(T))$. Since $\Ic I_h(cw)=c\Ic w=0$ for any $c\in\R$, we set $c_K^{\ell} := |\omega_K|^{-1}\int_{\omega_K} \eta_\ell$ for $K\in \T_H$ and get
\begin{eqnarray}
\nonumber\lefteqn{\|\nabla \Ic I_h(\eta_\ell w)\|_{L^2(\Omega)}^2}\\
\nonumber&\overset{\eqref{e:cutoffH}}{=}& \underset{K\subset \overline{{U}_{k+1}(T)}\setminus U_{k-\ell-1}(T)}{\sum_{K\in\T_H:}} \left\|\nabla\Ic I_h\left(\left(\eta_\ell-c_K^{\ell} \right)w\right)\right\|_{L^2(K)}^2 \\
\nonumber&\overset{\eqref{e:intesth},\eqref{e:interr}}{\lesssim}& \underset{K\subset \overline{U_{k+1}(T)}\setminus U_{k-\ell-1}(T)}{\sum_{K\in\T_H:}} \left\|\nabla\left(\left(\eta_\ell- c_K^{\ell} \right)w\right)\right\|_{L^2(\omega_K)}^2 \\
\nonumber&\overset{\eqref{e:Vf}}{\lesssim}& \underset{K\subset \overline{U_{k+1}(T)}\setminus U_{k-\ell-1}(T)}{\sum_{K\in\T_H:}}  \left\| (\nabla \eta_\ell)( w - I_H w) \right\|_{L^2(\omega_K)}^2 + \left\| \left(\eta_\ell- c_K^{\ell} \right)\nabla w \right\|_{L^2(\omega_K)}^2 \\
\nonumber&\overset{\eqref{e:cutoffH}}{\lesssim}& \underset{K\subset \overline{U_{k}(T)}\setminus U_{k-\ell}(T)}{\sum_{K\in\T_H:}} \hspace{-10pt} \left\| (\nabla \eta_\ell)( w - I_H w) \right\|_{L^2(K)}^2 \hspace{10pt}+ \hspace{-20pt} \underset{K\subset \overline{U_{k+1}(T)}\setminus U_{k-\ell-1}(T)}{\sum_{K\in\T_H:}} \hspace{-30pt} \left\| \left(\eta_\ell- c_K^{\ell} \right)\nabla w \right\|_{L^2(\omega_K)}^2 \\
\nonumber&\lesssim& \|H \nabla \eta_\ell\|^2_{L^\infty(\Omega)} \| \nabla w  \|_{L^2(U_{k+1}(T)\setminus U_{k-\ell-1}(T))}^2 + \hspace{-20pt} \underset{K\subset \overline{U_{k+1}(T)}\setminus U_{k-\ell-1}(T)}{\sum_{K\in\T_H:}}\hspace{-30pt} \left\| \left(\eta_\ell- c_K^{\ell} \right)\nabla w \right\|_{L^2(\omega_K)}^2\\
\label{e:ferlg}&{\lesssim}& \|H \nabla \eta_\ell\|^2_{L^\infty(\Omega)} \| \nabla w  \|_{L^2(U_{k+2}(T)\setminus U_{k-\ell-2}(T))}^2.
\end{eqnarray}
In the last step, we used the Lipschitz bound
\begin{align*}
 \| \eta_\ell- c_K^{\ell} \|_{L^{\infty}(\omega_K)}^2 \lesssim 
 H^2 \| \nabla \eta_\ell \|_{L^{\infty}(\omega_K)}^2.
\end{align*}
In summary we get with the previous computations:
\begin{eqnarray*}
\lefteqn{\|\nabla(\eta_\ell w-\tilde{w})\|_{L^2(\Omega)}^2\overset{\eqref{e:lackproj}}{\lesssim} \|\nabla(\eta_\ell w-I_h(\eta_\ell w))\|_{L^2(\Omega)}^2 + \|\nabla \Ic I_h(\eta_\ell w)\|_{L^2(\Omega)}^2}\\
&\overset{\eqref{e:intesth},\eqref{e:ferlg}}{\lesssim}&\|h\nabla\eta_\ell\cdot\nabla w\|^2_{L^2(\Omega)}+\|H \nabla \eta_\ell\|^2_{L^\infty(\Omega)} \| \nabla w  \|_{L^2(U_{k+2}(T)\setminus U_{k-\ell-2}(T))}^2 \\
&\overset{\eqref{e:cutoffH}}{\lesssim}& \left(\|h\nabla\eta_\ell\|^2_{L^\infty(\Omega)}+\|H \nabla \eta_\ell\|^2_{L^\infty(\Omega)}\right) \|\nabla w\|^2_{L^2(U_{k+2}(T)\setminus U_{k-\ell-2}(T))}\\
&\overset{\eqref{e:cutoffH}}{\lesssim}& \ell^{-2}\|\nabla w\|^2_{L^2(U_{k+2}(T)\setminus U_{k-\ell-2}(T))}.
\end{eqnarray*}
This proves the assertion.
\end{proof}

\begin{proof}[Proof of Lemma~\ref{l:decay}]
The proof exploits some recursive Caccioppoli argument as in \cite{2011arXiv1110.0692M}.
We fix some $T\in\T_H$ and $k\in\mathbb{N}$. Given $\ell\in\mathbb{N}$ with $\ell< k-1$, let
$\eta_\ell:=\eta_{T,k-2,\ell-4}\in V_H$ 
be some cutoff function as in \eqref{e:cutoffH}. Lemma~\ref{l:cutoff} shows that there exists some $\tilde{w}_T^i\in \Vf$ such that $\|\nabla(\eta_\ell w_T^i-\tilde{w}_T^i)\|_{L^2(\Omega)}\lesssim  \ell^{-1} \|\nabla w_T^i\|_{L^2(U_{k}(T)\setminus U_{k-\ell}(T))}$. Since $\tilde{w}_T^i \in \mathring{W}_h(\Omega \setminus U_{k-\ell+1}(T))$ and, hence, $\tilde{w}_T^i\vert_T=0$, it holds
\begin{align}
\label{lemm-4-4-step-1} \int_{\Omega \setminus U_{k-\ell}(T)} A \nabla w_T^i \cdot \nabla \tilde{w}_T^i =  \int_{\Omega} A \nabla w_T^i \cdot \nabla \tilde{w}_T^i = - \int_{T} A e_i \cdot \nabla \tilde{w}_T^i = 0.
\end{align}
The definition of $\eta_\ell$, the product rule, \eqref{e:correctorlocal}, and \eqref{e:Vf} yield
\begin{equation*}
 \begin{aligned}
  \int_{\Omega\setminus U_k(T)}A\nabla w_T^i\cdot \nabla w_T^i &\leq 
  \int_{\Omega\setminus U_{k-\ell}(T)}\eta_\ell A\nabla w_T^i\cdot \nabla w_T^i\\
  &= \int_{\Omega\setminus U_{k-\ell}(T)} A\nabla w_T^i\cdot \left(\nabla (\eta_\ell w_T^i)-w_T^i\nabla\eta_\ell\right)\\
&\overset{(\ref{lemm-4-4-step-1})}{=} \int_{\Omega\setminus U_{k-\ell}(T)} A\nabla w_T^i\cdot \left(\nabla (\eta_\ell w_T^i-\tilde{w}_T^i)-(w_T^i-\underset{=0}{\underbrace{\Ic(w_T^i)}})\nabla\eta_\ell\right)
 \end{aligned}
\end{equation*}
Observe that, by \eqref{e:cutoffH}, $\|\nabla\eta_\ell\|_{L^{\infty}(K)}=|\nabla\eta_\ell(x_K)|\lesssim \ell^{-1}H_K^{-1}$ for all $K\in\T_H$. This and the estimate \eqref{e:interr} for the interpolation error show that
\begin{equation*}
 \begin{aligned}
 \|(w_T^i-\Ic(w_T^i))\nabla\eta_\ell\|_{L^2(K)}^2&\lesssim \|\nabla\eta_\ell\|_{L^{\infty}(K)}^2\|w_T^i-\Ic(w_T^i)\|_{L^2(K)}^2\\
&\lesssim H_K^2\|\nabla\eta_\ell\|_{L^{\infty}(K)}^2\|\nabla w_T^i\|_{L^2(\omega_K)}^2\\
&\lesssim \ell^{-2}\|\nabla w_T^i \|_{L^2(\omega_K)}^2
 \end{aligned}
\end{equation*}
for any $K\in\T_H$.
The combination of the previous estimates and Cauchy-Schwarz inequalities prove that there is some constant $C_1>0$ independent of $T$, $\ell$, $k$, and the oscillations of $A$ such that 
\begin{equation}\label{e:decay}
\|A^{1/2}\nabla w_T^i \|_{L^2(\Omega\setminus U_{k}(T))}\leq C_1 \ell^{-1} \|A^{1/2}\nabla w_T^i \|_{L^2(\Omega\setminus U_{k-\ell-1}(T))}
\end{equation}
The choice
$\ell:=\left\lceil C_1 \hspace{2pt} e \right\rceil$
and the recursive application of \eqref{e:decay} readily yield the assertion. 
\end{proof}
This exponential decay justifies the approximation of the correctors on local patches $U_k(T)$ as proposed in (\ref{def-patch-U-k}). 
We denote by $\Qh^k$ the corrector that corresponds to the choice $U(T)=U_k(T)$ in Oversampling Strategy 3 and by $\Qh^\Omega$ the one for $U(T)=\Omega$.
\begin{corollary}[Truncation/Localization error]\label{c:truncation}
Let $U(T)=\Omega$ for all $T\in\T_H$ in Oversampling Strategy 3. Then, for all $T\in\T_H$ and all $k\in\mathbb{N}$,
$$\|A^{1/2}\nabla (w^i_T-w^{i,k}_T)\|_{L^2(\Omega\setminus U_k(T))}\lesssim e^{-r \cdot k}\|A^{1/2} e_i\|_{L^2(T)},$$
where $r > 0$ is as in Lemma \ref{l:decay} (independent of the variations of $A$ or the mesh size).
\end{corollary}
\begin{proof}
Galerkin orthogonality yields 
\begin{equation*}
 \|A^{1/2}\nabla (w^i_T-w^{i,k}_T)\|_{L^2(\Omega)}^2\leq \|A^{1/2}\nabla (w^i_T-\tilde{w})\|_{L^2(U_{k-1}(T))}^2+\|A^{1/2}\nabla w^i_T\|_{L^2(\Omega\setminus U_{k-1}(T))}^2, 
\end{equation*}
where $\tilde{w}\in \Vf$ is the finescale function that corresponds to $(1-\eta_{T,k-1,1})w^i_T$ and which is constructed in the same way as $\tilde{w}$ in the proof of Lemma~\ref{l:cutoff}. Here, $\eta_{T,k-1,1}$ is some cut-off function as in \eqref{e:cutoffH}. Since $\operatorname{supp}(\tilde{w})\subset\operatorname{supp}((1-\eta_{T,k-1,1})w^i_T)\subset U_{k-1}(T)$ we have that $\tilde{w}\in \oVf(U_k(T))$ and the use of Galerkin orthogonality is justified.
Proceeding as in Lemma~\ref{l:cutoff} shows that
\begin{equation*}
  \|A^{1/2}\nabla (w^i_T-w^{i,k}_T)\|_{L^2(\Omega)}^2\lesssim \|A^{1/2}\nabla w^i_T\|_{L^2(\Omega\setminus U_{k-2}(T))}^2
\end{equation*}
and the application of Lemma~\ref{l:decay} yields the assertion. 
\end{proof}
The proof of the main theorem requires one some technical result.
\begin{lemma}
\label{lemma-influence-intersections}Let $k\in \mathbb{N}_{>0}$ and let $\Phi_H \in V_H$, then
\begin{equation}
\label{equation-influence-intersections}\left\| (Q_h^\Omega-Q_h^k)\Phi_H\right\|_{L^2(\Omega)}^2\lesssim k^d\sum_{T\in\T_H}\sum_{i=1}^d|\partial_{x_i} \Phi_H(x_T)|^2\left\|A^{\frac{1}{2}}\nabla (w_T^i-w_T^{i,k})\right\|_{L^2(\Omega)}^2.
\end{equation}
\end{lemma}
\begin{proof}
Let $\eta_{T,k,1}$ be defined according to \eqref{e:cutoffH} and define $z:=(Q_h^\Omega-Q_h^k)\Phi_H \in \Vf$. We decompose the error as follows:
\begin{eqnarray*}
\lefteqn{\left\| (Q_h^\Omega-Q_h^k)\Phi_H\right\|_{L^2(\Omega)}^2}\\
&= &a(z,z) \\&=& \sum_{T\in\T_H}\sum_{i=1}^d\partial_{x_i} \Phi_H(x_T) a (w_T^i-w_T^{i,k},z (1- \eta_{T,k,1})) + \sum_{T\in\T_H}\sum_{i=1}^d\partial_{x_i} \Phi_H(x_T) a (w_T^i-w_T^{i,k},z \eta_{T,k,1})\\
&=:& \mbox{I} + \mbox{II}.
\end{eqnarray*}
For the first term we get
\begin{align*}
|\mbox{I}| &\lesssim \sum_{T\in\T_H}\sum_{i=1}^d |\partial_{x_i} \Phi_H(x_T)| \| A^{\frac{1}{2}} \nabla (w_T^i-w_T^{i,k}) \|_{L^2(\Omega)} \| \nabla \left( z (1- \eta_{T,k,1}) \right) \|_{L^2(U_{k+1}(T))},
\end{align*}
where with $I_H(z)=0$
\begin{align*}
\| \nabla \left( z (1- \eta_{T,k,1}) \right) \|_{L^2(U_{k+1}(T))} &\le \| \nabla z \|_{L^2(U_{k+1}(T))} + \| z \nabla \left( 1- \eta_{T,k,1} \right) \|_{L^2(U_{k+1}(T) \setminus U_{k}(T) )}\\
&\lesssim \| \nabla z \|_{L^2(U_{k+1}(T))} + \frac{1}{H} \| z - I_H(z) \|_{L^2(U_{k+1}(T) \setminus U_{k}(T) )}\\
&\lesssim  \| \nabla z \|_{L^2(U_{k+2}(T))} 
\end{align*}
and therefore
\begin{align*}
|\mbox{I}| &\lesssim \sum_{T\in\T_H}\sum_{i=1}^d|\partial_{x_i} \Phi_H(x_T)| \| A^{\frac{1}{2}} \nabla (w_T^i-w_T^{i,k}) \|_{L^2(\Omega)} \| \nabla z \|_{L^2(U_{k+2}(T))} \\
&\lesssim k^{\frac{d}{2}} \left( \sum_{T\in\T_H}\sum_{i=1}^d|\partial_{x_i} \Phi_H(x_T)|^2 \| A^{\frac{1}{2}} \nabla (w_T^i-w_T^{i,k}) \|_{L^2(\Omega)}^2 \right)^{\frac{1}{2}} \| \nabla z \|_{H^1(\Omega)}. 
\end{align*}
To estimate the second term, we use Lemma \ref{l:cutoff} which gives us the existence of some $\tilde{z}\in \mathring{W}_h(\Omega \setminus U_{k-2}(T))$ with $a(w_T^i-w_T^{i,k},\tilde{z})=0$ (as in \eqref{lemm-4-4-step-1}) and $\| \nabla (z \eta_{T,k,1} - \tilde{z}) \|_{L^2(\Omega)} \lesssim \| \nabla z \|_{L^2( U_{k+2}(T) )}$. This yields 
\begin{align*}
|\mbox{II}| &= \left|
\sum_{T\in\T_H}\sum_{i=1}^d\partial_{x_i} \Phi_H(x_T) a (w_T^i-w_T^{i,k}, z \eta_{T,k,1} - \tilde{z})
\right|\\
&\le \sum_{T\in\T_H}\sum_{i=1}^d |\partial_{x_i} \Phi_H(x_T)|  \| A^{\frac{1}{2}} \nabla (w_T^i-w_T^{i,k}) \|_{L^2(\Omega)} \| \nabla z \|_{L^2( U_{k+2}(T) )}\\
&\lesssim k^{\frac{d}{2}} \left( \sum_{T\in\T_H}\sum_{i=1}^d|\partial_{x_i} \Phi_H(x_T)|^2 \| A^{\frac{1}{2}} \nabla (w_T^i-w_T^{i,k}) \|_{L^2(\Omega)}^2 \right)^{\frac{1}{2}} \| \nabla z \|_{H^1(\Omega)}. 
\end{align*}
Combining the estimates for I and II and dividing by $ \| \nabla z \|_{H^1(\Omega)} \lesssim a(z,z)^{\frac{1}{2}}$ yields the assertion.
\end{proof}

\begin{theorem}[$H^1$ error estimate]\label{t:H1}
Given $k\in\mathbb{N}$, let $U(T)=U_k(T)$ for all $\T\in\T_H$ in Oversampling Strategy 3. Then the multiscale approximation $u_H^k$ that solves \eqref{e:MsVMMsym} satisfies the error estimate
\begin{equation*}
 \|\nabla u_h-\nabla(u_H^k+\Qh^k(u_H^k))\|\lesssim \gmin^{-1}\|H f\|_{L^2(\Omega)} + e^{-r  k}\|f\|_{H^{-1}(\Omega)},
\end{equation*}
where $u_h$ is the reference solution from \eqref{e:modelref} and $r>0$ as in Lemma \ref{l:decay}.
\end{theorem}
\begin{remark}[Relation to the results \cite{2011arXiv1110.0692M}]
In the case of maximal oversampling, the new MsFEM with constrained oversampling coincides with the ideal version (without localization) of the variational multiscale method presented in \cite{2011arXiv1110.0692M}. The localized versions are different and allow similar, but not identical error estimates. The upper bound obtained in \cite{2011arXiv1110.0692M}  reads (up to some  multiplicative constant) 
$$\|H f\|_{L^2(\Omega)} + H^{-1}e^{-r  k}\|f\|_{H^{-1}(\Omega)}.$$
Our new localization strategy allows for an improved estimate in the sense that the unpleasant factor $H^{-1}$ does not appear. Note that the proof of the error estimate in Theorem \ref{t:H1} does not generalize to the localization strategy used in \cite{2011arXiv1110.0692M} and must therefore be seen independently. The reason is that the structure of the local problems (\ref{e:correctorlocal}) gives us a nice summation property which we were able to exploit, but which is not available in  \cite{2011arXiv1110.0692M}.
This observation indicates that better numerical approximations for equal sizes of oversampling patches are possible with our new approach. This will be investigated in future works. 
\end{remark}

\begin{proof}[Proof of Theorem~\ref{t:H1}]
Using the fact that Galerkin approximation minimize the error in the energy norm, we obtain with the definitions of $u_H^k$ and $u_h$ that for all $\Phi_H \in V_H$
\begin{align}
\label{first-step-proof-thm-4-12} \|&A^{1/2}(\nabla u_h-\nabla(u_H^k+\Qh^k(u_H^k)))\|_{L^2(\Omega)}\leq \|A^{1/2}(\nabla u_h-\nabla(\Phi_H+\Qh^k(\Phi_H)))\|_{L^2(\Omega)}.
\end{align}
Let $u_H$ be the solution of \eqref{e:MsVMMsym} with the ideal corrector $\Qh=\Qh^\Omega$. Then
\begin{eqnarray*}
\lefteqn{\|A^{1/2}(\nabla u_h-\nabla(u_H^k+\Qh^k(u_H^k)))\|_{L^2(\Omega)}}\\
&\overset{(\ref{first-step-proof-thm-4-12})}{\leq}& \|A^{1/2}(\nabla u_h-\nabla(u_H+\Qh^k(u_H)))\|_{L^2(\Omega)}\\
&\leq& \|A^{1/2}(\nabla u_h-\nabla(u_H+\Qh^\Omega(u_H)))\|_{L^2(\Omega)}\\
&\enspace&\qquad+\|A^{1/2}(\nabla(u_H+\Qh^\Omega(u_H))-\nabla(u_H+\Qh^k(u_H)))\|_{L^2(\Omega)}\\
&\lesssim& \gmin^{-1/2}\|H f\|_{L^2(\Omega)} +\|A^{1/2}\nabla((\Qh^\Omega-\Qh^k)(u_H))\|_{L^2(\Omega)}.
\end{eqnarray*}
By Corollary~\ref{c:truncation}, we get
\begin{eqnarray*}
\lefteqn{\|A^{1/2}\nabla((\Qh^\Omega-\Qh^k)(u_H))\|_{L^2(\Omega)}^2}\\
&=&\left\| \sum_{T\in\T_H}\sum_{i=1}^d\partial_{x_i} u_H(x_T)A^{\frac{1}{2}}\nabla (w_T^i-w_T^{i,k})\right\|_{L^2(\Omega)}^2\\
&\overset{\eqref{equation-influence-intersections}}{\lesssim}&
k^d\sum_{T\in\T_H}\sum_{i=1}^d|\partial_{x_i} u_H(x_T)|^2\left\|A^{\frac{1}{2}}\nabla (w_T^i-w_T^{i,k})\right\|_{L^2(\Omega)}^2\\
&\lesssim& k^d 
e^{-2r \cdot k}\sum_{T\in\T_H}\sum_{i=1}^d|\partial_{x_i} u_H(x_T)|^2\|A^{1/2}e_i\|_{L^2(T)}^2\\
&\lesssim& k^d 
e^{-2r \cdot k}\sum_{T\in\T_H}\sum_{i=1}^d\|A^{1/2}\nabla u_H\|_{L^2(T)}^2\\
&\lesssim& k^d 
e^{-2r k} \|f\|_{H^{-1}(\Omega)}^2.
\end{eqnarray*}
In the last step we have used that $u_H=\Ic(u_H+\Qh^\Omega(u_H))$, the stability of $\Ic$ and the energy estimate $\|A^{1/2}\nabla (u_H+\Qh^\Omega(u_H)\|_{L^2(T)}\lesssim \gmin^{-1/2}\|f\|_{H^{-1}(\Omega)}$.
\end{proof}

\begin{theorem}[$L^2$-estimates]\label{t:L2}
Given $k\in\mathbb{N}$, let $U(T)=U_k(T)$ for all $\T\in\T_H$ in Oversampling Strategy 3. Then the multiscale approximation $u_H^k$ that solves \eqref{e:MsVMMsym} satisfies the error estimates
\begin{equation*}
 \|u_h-(u_H^k+\Qh^k(u_H^k))\|_{L^2(\Omega)}\lesssim (\gmin^{-1}\|H\|_{L^{\infty}(\Omega)}+k^{d/2}
e^{-r k})^2\|f\|_{L^2(\Omega)}
\end{equation*}
and
\begin{equation*}
\|u_h-u_H^k\|_{L^2(\Omega)}
\lesssim \min_{v_H\in V_H}\|u_h-v_H\|_{L^2(\Omega)}+(\gmin^{-1}\|H\|_{L^{\infty}(\Omega)}+k^{d/2}
e^{-r k})^2\|f\|_{L^2(\Omega)},
\end{equation*}
where $u_h$ is the reference solution from \eqref{e:modelref} and $r$ is a positive constant.
\end{theorem}
\begin{proof}
A standard Aubin-Nitsche duality argument yields the first estimate.
The second estimate follows from the first one and the quasi-optimality and stability of the interpolation $\Ic$ in $L^2(\Omega)$.
\end{proof}

\begin{remark}[Smooth coefficient with known smallest scale $\varepsilon$.]
Let $\Omega$ be convex, $f\in L^2(\Omega)$ with $\|f\|_{L^2(\Omega)}\lesssim1$, $A\in W^{1,\infty}(\Omega)$ with $\|\nabla A\|_{L^{\infty}(\Omega)}\lesssim \varepsilon^{-1}$ with some small scale parameter $\varepsilon>0$. Choose uniform meshes $\T_H$ and $\T_h$ with $H\gtrsim \varepsilon\gtrsim h$. 
Under these assumptions, the error of the reference solution $u_h\in V_h$ is bounded as follows,
$$\|\nabla(u-u_h)\|\lesssim h \varepsilon^{-1}.$$
We refer to see \cite{PS12} for details. 
If $k\gtrsim\log(H^{-1})$, Theorem~\ref{t:H1} and Theorem~\ref{t:L2} yield the error bounds
\begin{eqnarray*}
\|\nabla(u-u_H^k-\Qh^k(u_H^k))\|_{L^2(\Omega)}&\lesssim& H+\tfrac{h}{\varepsilon},\\
 \|u-u_H^k-\Qh^k(u_H^k)\|_{L^2(\Omega)}&\lesssim& H^2+\left(\tfrac{h}{\varepsilon}\right)^2,\\
 \|u-u_H^k\|_{L^2(\Omega)}&\lesssim& H+\left(\tfrac{h}{\varepsilon}\right)^2.
\end{eqnarray*}
\end{remark}

\section{Numerical experiments}
\label{section-numerical-experiments}

\begin{table}[t]
\caption{Computations made for $h=2^{-6}$. $k$ denotes the number of coarse layers. $u_h$ denotes the fine scale reference given by  (\ref{e:modelref}) and $u^{\mathbf{3},\mbox{\rm \tiny MsFEM}}_H$ denotes the MsFEM approximation obtained with Oversampling Strategy 3. The table depicts various errors between $u_h$ and $u^{\mathbf{3},\mbox{\rm \tiny MsFEM}}_H$.}
\label{serie-convergence}
\begin{center}
\begin{tabular}{|c|c|c|c|c|}
\hline $H$      &  \mbox{Fine layers} &
          \mbox{k} & $\|u_h - u^{\mathbf{3},\mbox{\rm \tiny MsFEM}}_H\|_{L^2(\Omega)}$ & $\|u_h - u^{\mathbf{3},\mbox{\rm \tiny MsFEM}}_H\|_{H^1(\Omega)}$ \\
\hline $2^{-1}$ & 16 & 0.5      & 0.490063 & 4.49575 \\
\hline
\hline $2^{-2}$ & 8   & 0.5    & 0.09491 & 1.66315 \\
\hline $2^{-2}$ & 24 & 1.5    & 0.06376 & 1.08960 \\
\hline
\hline $2^{-3}$ & 4    & 0.5 &  0.033691 & 1.017150 \\
\hline $2^{-3}$ & 8    & 1    &  0.007125 & 0.406317 \\
\hline $2^{-3}$ & 12  & 1.5 &  0.007115 & 0.331458 \\
\hline $2^{-3}$ & 16  & 2     & 0.003241 & 0.165703  \\
\hline
\hline $2^{-4}$ &  2    & 0.5 & 0.012808 & 0.655269 \\
\hline $2^{-4}$ &  4    & 1    & 0.004164 & 0.348814 \\
\hline $2^{-4}$ &  6    & 1.5 & 0.004029 & 0.329306 \\
\hline $2^{-4}$ &  8    & 2    & 0.001451 & 0.162747 \\
\hline $2^{-4}$ &  12  & 2.5 & 0.000850 & 0.114040 \\
\hline $2^{-4}$ &  16  & 3    & 0.000696 & 0.096378 \\
\hline
\end{tabular}
\end{center}
\end{table}

\begin{table}[h]
\caption{Computations made for $h=2^{-6}$. $k$ denotes the number of coarse layers. $u_h$ denotes the fine scale reference given by  (\ref{e:modelref}) and $u^{\mathbf{i},\mbox{\rm \tiny MsFEM}}_H$ denotes the MsFEM approximation obtained with Oversampling Strategy $i$. The error is denoted by $e_{i}:=u_h - u^{\mathbf{i},\mbox{\rm \tiny MsFEM}}_H$. The second column depicts the number of fine grid layers.}
\label{serie-comparison-strategies}
\begin{center}
\begin{tabular}{|c|c|c|c|c|c|c|c|c|c|}
\hline \multicolumn{2}{|c|}{$\enspace$} & \multicolumn{2}{|c|}{Strategy 1} & \multicolumn{2}{|c|}{Strategy 2} &\multicolumn{2}{|c|}{Strategy 3}\\
\hline $H$      &  $k$ & $\|e_1\|_{L^2}$ & $\|e_1\|_{H^1}$
& $\|e_2\|_{L^2}$ & $\|e_2\|_{H^1}$
& $\|e_3\|_{L^2}$ & $\|e_3\|_{H^1}$ \\
\hline
\hline $2^{-2}$ & 1 &
   0.1399      & 1.9812  &
   0.1399      & 1.9812  &
   0.0638      & 1.0896  \\
\hline
\hline $2^{-3}$ &  1   &
   0.0594      & 1.6250  &
   0.0594      & 1.6250  &
   0.0071      & 0.4063  \\
\hline $2^{-3}$ &  2 &
   0.0593           & 1.6250  &
   0.0593           & 1.6250  &
   0.0032           & 0.1657  \\
\hline
\hline $2^{-4}$ &  1   &
   0.0166      & 0.8067  &
   0.0172      & 0.8048  &
   0.0042      & 0.3488 \\
\hline $2^{-4}$ &  2   &
   0.0160           & 0.8057  &
   0.0168           & 0.7955  &
   0.0015           & 0.1628 \\
\hline $2^{-4}$  &  3 &
   0.0153                             & 0.8016  &
   0.0152                             & 0.7937  &
   0.0007                             & 0.0964  \\
\hline
\end{tabular}
\end{center}
\end{table}

In this section we present numerical experiments to confirm the derived error estimates and to compare the numerical accuracies of the oversampling strategies 1, 2 and 3. Here we use strategy 1 and 2 in Petrov-Galerkin formulation (due to the findings in \cite{Hou:Wu:Zhang:2004}) and strategy 3 in symmetric formulation. We consider the following model problem.
\begin{modelproblem}
Let $\Omega := ]0,1[^2$ and $\epsilon=5 \cdot 10^{-2}$. We define
\begin{align*}
 u(x_1,x_2):= \mbox{\rm sin}( 2 \pi x_1 ) \mbox{\rm sin}( 2 \pi x_2 ) + \frac{\varepsilon}{2} \mbox{\rm cos}( 2 \pi x_1 ) \mbox{\rm sin}( 2 \pi x_2 ) \mbox{\rm sin}( 2 \pi \frac{x_1}{\varepsilon} ),
\end{align*}
which is the exact solution of the problem
\begin{align*}
- \nabla \cdot \left( A \nabla u \right) &= f \quad \mbox{in} \enspace \Omega \\
 u &= 0 \quad \mbox{on} \enspace \partial \Omega,
\end{align*}
where $A$ is given by
\begin{eqnarray*}
A(x_1,x_2):= \frac{1}{8 \pi^2} \left(\begin{matrix}
                         2(2 + \mbox{\rm cos}( 2 \pi \frac{x_1}{\varepsilon} ))^{-1}  & 0 \\
                         0 & 1 + \frac{1}{2}\mbox{\rm cos}( 2 \pi \frac{x_1}{\varepsilon} )
                        \end{matrix}\right)
\end{eqnarray*}
and $f$ by
\begin{align*}
 f(x):= - \nabla \cdot \left( A(x) \nabla u(x) \right) \approx \mbox{\rm sin}( 2 \pi x_1 ) \mbox{\rm sin}( 2 \pi x_2 ).
\end{align*}
\end{modelproblem}
In Table \ref{serie-convergence} we depict the results for $h=2^{-6}$ and various combinations of $H$ with different numbers of oversampling layers. For a better illustration we state the number of fine grid layers and the number of coarse grid layers ($k$) that corresponds with that. The results in Table \ref{serie-convergence} match nicely with the analytically predicted behavior. In Table \ref{serie-comparison-strategies} we state a comparison between the $L^2$- and $H^1$-errors for the three oversampling strategies obtained for identical values of $H$, $h$ and $\mathcal{U}$. We observe that our oversampling strategy, in contrast to the classical ones, does not suffer from a loss in accuracy when $H$ is close to the microscopic parameter $\varepsilon$. Moreover, the  accuracy obtained for strategy 3 is very promising in general.

\section{Conclusion}\label{section-conclusion}

In this work, we proposed a new oversampling strategy for the Multiscale Finite Element Method (MsFEM) which generalizes the original method without oversampling. The new strategy is based on an additional constrained for the solution spaces of the local problems. The error analysis shows that oversampling layers of thickness $H \log(H^{-1})$ suffice to preserve the common convergence rates with respect to $H$ without any preassymptotic effects. Moreover, this choice prevents resonance errors even for general $L^\infty$ coefficients without any assumptions on the geometry of the microstructure or the regularity of $A$. In this respect, the method is reliable. The method is also efficient in the sense that structural knowledge about the coefficient, e.g. (local) periodicity or scale separation, may be exploited to reduce the number of corrector problems considerably. Whether the oversampling can be reduced to very small layers in the case of e.g. periodicity, should be investigated numerically and/or analytically in future works.

$\\$
{\bf{Acknowledgements.}} The authors gratefully acknowledge the helpful suggestions made by the anonymous referees, which greatly improved the presentation of the paper. Further more we thank Antoine Gloria for pointing out useful references and Xiao-Hui Wu for the nice discussions at the SIAM Geoscience Conference 2013 and his helpful remarks.


\appendix

\bibliographystyle{alpha}

\end{document}